\documentclass[12pt]{amsart}
\usepackage[latin9]{inputenc}
\usepackage{color}
\usepackage{amsfonts}
\usepackage{amssymb}
\usepackage{amsmath}
\usepackage{amsthm}
\usepackage{geometry}
\usepackage{verbatim}
\usepackage[toc,page]{appendix}
\usepackage[backref=page]{hyperref}
\hypersetup{backref=true}
\usepackage[all,cmtip]{xy}

\RequirePackage{xstring}

\DeclareMathOperator\Gal{Gal}
\DeclareMathOperator\Aut{Aut}
\DeclareMathOperator\Out{Out}

\DeclareMathOperator\core{core}

\DeclareMathOperator\PGammaL{P\Gamma L}

\DeclareMathOperator\Sym{Sym}

\DeclareMathOperator\Image{Im}
\DeclareMathOperator\HG{H}

\DeclareMathOperator\Ind{Ind}
\DeclareMathOperator\Bad{Bad}

\newcommand{\Z}{{\mathbb Z}} 
\newcommand{\Q}{{\mathbb Q}}
\newcommand{\C}{{\mathbb C}}

\newcommand{\firstCondition}[4]{\IfEq{#4}{C}{There}{there} exists a \ur{} decomposition ${#1}:= {#2} \circ {#3}$ of ${#1}$ where ${#2},{#3} \in K[X]$, and ${#3}$ is an indecomposable polynomial which is neither cyclic nor dihedral}

\DeclareMathOperator\soc{soc}
\DeclareMathOperator\mon{Mon}
\DeclareMathOperator\Mon{Mon}
\DeclareMathOperator\Br{Br}
\DeclareMathOperator\Ram{Ram}

\begin{document}
\providecommand{\keywords}[1]{\textbf{\textit{Keywords: }} #1}
\newtheorem{theorem}{Theorem}[section]
\newtheorem{lemma}[theorem]{Lemma}
\newtheorem{prop}[theorem]{Proposition}
\newtheorem{cor}[theorem]{Corollary}
\newtheorem{problem}[theorem]{Problem}
\newtheorem{question}[theorem]{Question}
\newtheorem{conjecture}[theorem]{Conjecture}
\newtheorem{claim}[theorem]{Claim}
\newtheorem{condition}[theorem]{Condition}

\theoremstyle{definition}
\newtheorem{defn}[theorem]{Definition} 
\theoremstyle{remark}
\newtheorem{remark}[theorem]{Remark}
\newtheorem{example}[theorem]{Example}
\newtheorem{condenum}{Condition}

\newcommand{\cc}{{\mathbb{C}}}
\newcommand{\mQ}{{\mathbb{Q}}}
\newcommand{\nn}{{\mathbb{N}}}
\newcommand{\qq}{{\mathbb{Q}}}
\newcommand{\rr}{{\mathbb{R}}}
\newcommand{\mP}{{\mathbb{P}}}
\newcommand{\zz}{{\mathbb{Z}}}
\newcommand{\fp}{{\mathfrak{p}}}
\newcommand{\ra}{{\rightarrow}}
\newcommand{\eps}{{\varepsilon}}
\newcommand{\divides}{\,|\,}

\newcommand\DN[1]{{\color{black}{#1}}}
\newcommand\JK[1]{{\color{black} {#1}}}
\newcommand\oline[1] {{\overline{#1}}}

\def\technion{Department of Mathematics, Technion - Israel Institute of Technology, Haifa, Israel}
\def\weizmann{Department of Mathematics, Weizmann Institute of Science, Rehovot, Israel}

\title[Monodromy groups of polynomial compositions]{Polynomial compositions with large monodromy groups and applications to arithmetic dynamics}

\author{Joachim K\"onig}
\address{Department of Mathematics Education, Korea National University of Education, Cheongju, South Korea}
\email{jkoenig@knue.ac.kr}

\author{Danny Neftin}
\address{\technion}
\email{dneftin@technion.ac.il}%

\author{Shai Rosenberg}
\address{\technion}
\email{shai.ros@alumni.technion.ac.il}
\begin{abstract}
For a composition $f=f_1\circ\cdots \circ f_r$ of polynomials $f_i\in \mQ[x]$ of degrees $d_i\geq 5$ with alternating or symmetric monodromy group, we show that the monodromy group of $f$ contains the iterated wreath product $A_{d_r}\wr \cdots\wr A_{d_1}$. A similar property holds more generally for polynomials that do not factor through $x^d$ or Chebyshev. We derive consequences to arithmetic dynamics regarding arboreal representations, and forward and backward orbits of such $f$. In particular, given an orbit $(a_n)_{n=0}^\infty$ of $f$ as above, we show
that for ``almost all" $a\in \Z$, the set  of primes $p$ for which some $a_n$ is congruent to $a$ mod $p$ 
is ``small".
\end{abstract}
\maketitle

\section{Introduction}\label{sec: intro - methods}

The study of monodromy groups $\Mon(f):=\Gal(f(x)-t,\mQ(t))$  of polynomial maps $\mP^1_\mQ\ra \mP^1_\mQ$, $x\mapsto f(x)$ for $f\in \mQ[x]$, 
lies at the heart of many topics in number theory, dynamics, and other subjects. These  include  Hilbert sets \cite{Mul3,Fried,NZ,KN}, functional decompositions \cite{AZ,MZ,Pak2} and reducibility of iterates \cite{Ferr, Pak}. Moreover, 
monodromy groups of iterates $f^{\circ n}$ are of major significance in arithmetic dynamics, cf.\ \cite{Jones-mon,Nak}, in view of their following  close relation to {arboreal Galois representations (of the absolute Galois group of $\mQ$)}, an object of central role in the subject \cite{BIJMST,Jones-arbor, BJ}.  

The images 
$\Image\rho^{(n)}_{f,a}:=\Gal(f^{\circ n}(x)-a,\mQ)$ of  arboreal representations embed into $\Mon(f^{\circ n})=\Gal(f^{\circ n}(x)-t,\mQ(t))$ for most specializations  $t\mapsto a\in \mQ$.
Since 
much information on forward and backward orbits of $f$ is encoded in the action of $\Image\rho_{f,a}^{(n)}$ 
on the tree  $T^{(n)}_{f,a}=\bigcup_{i=0}^n(f^{\circ n})^{-1}(a)$ of preimages of $a$, 
\JK{key questions about these orbits depend decisively} on the 
the action of $\Mon(f^{\circ n})$ on $T^{(n)}=T^{(n)}_{f,t}$. \JK{We give two examples in the following.}

\JK{In the context of forward orbits, the actions of 
 the groups $\Image\rho_{f,a}^{(n)}$, $n\in \mathbb{N}$ determine the (natural) density of primes $p$ such that, for prescribed $a_0$ and $a\in \mathbb{Q}$, some value $f^{\circ n}(a_0)$, $n\in \mathbb{N}$ is congruent\footnote{Congruence conditions require that $p$ is coprime to the denominators of both sides.} to $a$ mod $p$, i.e., such that $a$ meets the orbit of $a_0$ under $f$ modulo $p$.}
For example, when $f(x)$ is $x^2-x+1$ (resp.\ $x^2$), $a=0$ (resp.\ $a=-1$), and $a_0=2$, this is the density of primes dividing  some element in the (``Euclid type") Sylvester sequence $a_{n+1}=1+a_0\cdots a_n$ (resp.\ some Fermat number $f^{\circ n}(a_0)+1=2^{2^n}+1$). 
When $\Image \rho_{f,a}^{(n)}$ is the full group $\Aut(T^{(n)})$ for all $n$, or under mild conditions merely if its index in $\Aut(T^{(n)})$ is bounded independently of $n$,  
the above density is $0$ \cite[Theorems 4.1, 4.2]{Jones-arbor}. In particular, this density is $0$ for the above Euclid type sequence \cite{Odoni}. However, it is unknown for which $f\in\mQ[x]$ such largeness holds even for infinitely many $a\in \mQ$, cf.\  \cite[\S 5 and Conj.\ 5.5]{BIJMST}. 
Recall that  $\Aut(T^{(n)})$  is well known to be the $n$-fold iterated wreath product $[S_d]^n=S_d\wr \cdots \wr S_d$, where $S_d\wr S_d=S_d^d\rtimes S_d$ is the (standard imprimitive) wreath product and $d:=\deg f$.  


\JK{In the context of backward orbits, the actions of the groups $\Image\rho_{f,a}^{(n)}$ determine the density of primes $p$ such that all fibers $(f^{\circ n})^{-1}(a)$, $n\in \mathbb{N}$  
are irreducible mod $p$ (resp.\ have $\leq C$ irreducible components mod $p$ for a constant $C\in \mathbb N$).} 
These conditions are equivalent to $f^{\circ n}(x)-a$ mod $p$ being irreducible (resp.\ having $\leq C$ irreducible factors), and we then say $f$ mod $p$ is stable (resp.\ $C$-stable) over $a$. 
When all $\Image\rho_{f,a}^{(n)}$, $n\in \mathbb{N}$ are large in a suitable sense \cite{Ferr}, the density of primes $p$ for which $f$ mod $p$ is stable  over $a$ is  $0$. However in general, it is unknown whether  such largeness holds even for infinitely many  $a\in \mQ$.  
See   \cite[\S 9]{BIJMST} or \cite{BJ2, MOS} for further details on the topic. 

\JK{Much of the difficulty in 
solving problems such as the above \DN{for most (or even infinitely many)} $a\in \mQ$ lies in determining for which $f$ the groups  $\Mon(f^{\circ n})$, $n\in\mathbb N$  are sufficiently ``large".} 
The groups  $\Mon(f)$, $f\in \mQ[x]$ were classified for indecomposable $f$  \cite{Feit, Mul}, that is, when $f$ is not a composition of two maps of smaller degree. However, excluding families such as certain self-similar maps \cite{BN,Nak} and the normalized Belyi maps \cite{BEK,ABEGKM}, little is known about the possibilities for $\Mon(f)$, or about when  it is ``large",  for nontrivial composition $f=f_1\circ\cdots\circ f_r$ of polynomials $f_i\in \mQ[x]$.

In this paper, we show that $\Mon(f)$ is ``large" for  compositions $f=f_1\circ\cdots\circ f_r$ of indecomposable polynomials which are not {\it linearly related (over $\C$)} to $x^d$ or a Chebyshev polynomial $T_d$, $d\in\mathbb N$.\footnote{Here, two polynomials $f,g\in K[X]$ are called linearly related over $K$ if there exist linear polynomials $\lambda,\mu\in K[X]$ such that $f=\lambda\circ g\circ \mu$.}
When $\Mon(f_i)$ is alternating or symmetric, this largeness amounts to the following assertion: 
\begin{theorem}\label{thm:main-intro}
Suppose $f=f_1\circ \cdots\circ f_r$ for $f_i\in \mQ[x]$ of degree $d_i\geq 5$ with  $\Mon(f_i)\in\{A_{d_i},S_{d_i}\}$, $i=1,\ldots,r$. 
Then $\Mon(f)$ contains \DN{a subgroup isomorphic to}
$A_{d_r}\wr \cdots \wr A_{d_1}$. 
\end{theorem}
More generally when $f_i\in \mQ[x]$, $i=1,\ldots,r$ are indecomposable of degrees $d_i\geq 5$ and are not linearly related to $x^{d_i}$ or $T_{d_i}$ (over $\C$), the following holds for all $n\in \mathbb N$ by Corollary \ref{cor:pols iterative}: 
The group $\Mon(f_1\circ\dots\circ f_n)$ contains the (multiset of Jordan-H\"older) composition factors of $[\soc(\Gamma_i)]_{i=1}^n=\soc(\Gamma_n)\wr \cdots \wr \soc(\Gamma_1)$, where $\soc(
\Gamma_i)$ is the socle of $\Gamma_i:=\Mon(f_i)$, that is, the group generated by the minimal normal subgroups of $\Gamma_i$. When $\Gamma_i=A_{d_i}$ or $S_{d_i}$ as above, one has $\soc(\Gamma_i)=A_{d_i}$. Henceforth we use the following refined notion of largeness: For a decomposition $f=g\circ f_d$,  say that the kernel of  
the projection $\Mon(f)\ra\Mon(g)$ is {\it large} if it contains 
$\soc(\Gamma_d)^{\deg g}$. Thus $\Mon(f)$ contains the  composition factors of $[\soc(\Gamma_i)]_{i=1}^r$ if and only if the kernels of the projections 
$\Mon(f_1\circ\cdots\circ f_{i+1})\ra\Mon(f_1\circ\cdots \circ f_i)$ are large for $i=1,\ldots,r-1$. 

Applying this property to iterates of an indecomposable polynomial $f\in \mQ[x]$ of degree $d\geq 5$, not linearly related to $x^d$ or $T_d$, 
Hilbert's irreducibility theorem (HIT)  immediately yields 
the following corollary. Recall that a Hilbert subset of $\mQ$  
is the complement of  a union of finitely many value sets $g_i(X_i(\mQ))$ of maps $g_i:X_i\ra\mP^1_\mQ$. 
\begin{cor}\label{cor:spec}
Let $n\geq 2$ be an integer and $f\in \mQ[x]$ be indecomposable of degree $d\geq 5$ that is not linearly related to $x^d$ or $T_d$ over $\C$. Set $\Gamma:=\Mon(f)$. 
Then $\Image\rho_{f,a}^{(n)}$ contains the composition factors of $[\soc(\Gamma)]^n$ for all $a$ in a Hilbert subset of $\mQ$.
\end{cor}

These largeness properties for $\Mon(f^{\circ n})$ and $\Image \rho_{f,a}^{(n)}$ are compatible with bounding the above prime densities. For ``nonspecial" $f$ and most $a\in \Z$, in \S\ref{sec:fp-cycles} we show that both the density 
of primes for which $a$ meets an orbit of $f$ mod $p$,  and the density 
of primes $p$ for which the number of irreducible factors of  $f^{\circ n}(x)-a$ mod $p$ is bounded by a constant,  are arbitrarily small:
\begin{cor}\label{cor:epsilon}
Suppose $\eps>0$, $C\in \mathbb N$, and $f\in \Q[x]$ is indecomposable of degree $d\geq 5$  which is not linearly related to $x^d$ or $T_d$ over $\C$. 
Then: \\ 
(1) the density of primes $p$ for which $f^{\circ n}(a_0)\equiv a$ (mod $p$) for some $n\in \mathbb N$  is at most $\eps$,
for all $a_0\in \Z$ and $a$ in a Hilbert subset $H(f,\eps)\subset \mQ$; and \\ 
(2)  the density of primes $p$ for 
which 
$f$ mod $p$ is $C$-stable is at most $\eps$, for all 
$a$ in a Hilbert subset  $H(f,\eps,C)\subseteq \mQ$. 
\end{cor}
If $f$ is post-critically finite (PCF) polynomial, that is, if the number of branch points of $f^{\circ n}$ is uniformly bounded over all $n$, then  Corollary \ref{cor:epsilon}.(1) follows from \cite[Theorem 1.1]{Jones-mon} (or \cite{Jones-recent}) and HIT, unless  $f$ is of an explicit exceptional form. 

Note, however, that  
HIT cannot be applied simultaneously to infinitely many $n$, see \cite{DK} for examples over hilbertian fields. It is therefore  unknown  whether the largeness of $\Mon(f^{\circ n})$, $n\in\mathbb N$, in the sense of Theorem \ref{thm:main-intro}, implies the existence of infinitely many $a\in \mQ$ such that the kernel of the projection $\Image \rho_{f,a}^{(n)} \ra \Image \rho_{f,a}^{(n-1)}$ is {\it large} even merely for infinitely many $n$. Here the kernel is large if, as above, it contains $\soc(\Gamma)^{d^{n-1}}$, where $\Gamma:=\Mon(f)$ and $d:=\deg f$.  There are several indications though that this  might be the case. 
In particular,  it follows if   the index of $\Image \rho_{f,a}^{(n)}$  in $\Mon(f^{\circ n})$ is uniformly bounded for $n\in\mathbb N$, an assertion that is widely believed for non-PCF maps \cite[Question 1.1]{BDGHT} and holds for quadratic and cubic  maps  conditionally on the ABC, Vojta, and eventual-stability conjectures \cite{BT,BDGHT1, Hindes, JKLMTW}. The last conjecture asserts that $f$ is {\it eventually stable} over $a$, that is,  the  number of irreducible factors of $f^{\circ n}(x)-a$ is bounded as $n\to\infty$,  for all but finitely many $a\in\mathbb Q$ \cite{JL}. 

For PCF polynomials $f$, we show that for infinitely many $a\in \mQ$, the image {$\Image \rho_{f,a}^{(n)}$ indeed contains the composition factors of  $[\soc(\Gamma)]^n$ for all $n\in\mathbb N$,    under conditions that hold for a great variety of choices of $f$}. 
Infinite families of polynomials $f$ satisfying the conditions are given in  Example \ref{exam:shabat}, 
and the analogous assertion for more general dynamical systems is proved in Proposition \ref{prop:spec}. From a group-theoretic viewpoint, related to the notion of {\it invariable generation}, these conditions hold for most $f$, see Remark  \ref{rem:cond}.  
Moreover, recently \cite{BGJT} proved the assertion  for the opposite scenario when $\Gamma$ is a $p$-group.

When indeed the kernels of the projections $\Image \rho_{f,a}^{(n)}\to \Image \rho_{f,a}^{(n-1)}$ are {large}, even merely for infinitely many $n$, the compatibility with prime densities goes even further: 
\begin{theorem}
\label{cor:cycles}
Let $f\in \mathbb{Q}[x]$ be indecomposable of degree $d\ge 2$ and $C\in \mathbb N$ a constant. Suppose $a\in \mathbb{Q}$ is not a branch point\footnote{The assumption that $a$ is not a branch point of $f^{\circ n}$, $n\in\mathbb N$ is redundant for a suitable definition of an arboreal representation,  cf.\ \S\ref{sec:fp-cycles}.} of any $f^{\circ n}$, $n\in\mathbb N$,   and that $\ker(\Image\rho_{f,a}^{(n)}\to \Image\rho_{f,a}^{(n-1)})$ is large for infinitely many $n\in \mathbb{N}$.
 Then: \\
(1) the density of primes $p$ for which $f$ mod $p$ is $C$-stable over $a$ is $0$. \\ 
(2) Suppose additionally that $f$ is eventually stable over $a$,  not linearly related to $x^d$ or $T_d$, and $d\ge 5$.  
Then the density of primes for which $f^{\circ n}(a_0)\equiv a$ mod $p$ for some $n$ is $0$, for all $a_0\in \mQ$ such that $a\notin \{f^{\circ n}(a_0): n \in \mathbb{N}\}$. 
\end{theorem}
Part (1) of Theorem \ref{cor:cycles} extends \cite[Theorem 1.1]{Ferr} 
since the index growth rate assumed in  \cite{Ferr} does not allow large kernels in our sense.  Part (1) follows from  Proposition \ref{lem:cycles1} which  addresses $C$-stability for more general dynamical systems. 
Part (2) follows from Proposition \ref{lem:cycles2}. The additional eventual-stability assumption of this part, conjectured in \cite{JL},  is automatic when $\ker(\Image\rho_{f,a}^{(n)}\to \Image\rho_{f,a}^{(n-1)})$ is large for {\it all} $n$, since the fiber of $f^{\circ n}$ over $a$ is then irreducible  for all $n$.

\subsubsection*{Obstructions to largeness} The above Theorem \ref{thm:main-intro} and Corollary \ref{cor:pols iterative}  are consequences of the main theorem, Theorem \ref{thm:main}. As Theorem \ref{thm:main} is of purely group-theoretic nature, it applies more generally to decompositions $f=g\circ h$, for finite maps  $g,h$ between varieties over arbitrary fields of characteristic $0$, asserting that the kernel $K$ of the projection $\Mon(f)\ra\Mon(g)$ is large.  As in all of the above, it assumes $\soc(\Gamma)$ is a nonabelian minimal normal subgroup of $\Gamma:=\Mon(h)$. 
For such decompositions, we define four obstructions whose vanishing is necessary for the largeness of $K$, see \S\ref{sec:ker} and \S\ref{sec:nec}. 
Consequently, Theorem \ref{thm:main} assumes the  four obstructions vanish. 

The main obstruction, among the four, is the existence of Ritt moves for  decomposition of $f=g\circ h$, namely, the existence of another decomposition $f=g_2\circ h_2$ where $h_2$ and $h$ do not have a nontrivial common right factor, cf.\ the notion of invariant decompositions in \S\ref{sec:inv}.
The second and third obstructions concern trivial extensions of monodromy groups: namely, it is possible that the kernel $K$ of the projection $\Mon(f)\ra\Mon(g)$ is trivial, see Galois-proper decompositions in \S\ref{sec:ker}; and also that $\Mon(f)$ embeds into $\Aut(\Mon(h))$ forcing $K$ to be small, see \S\ref{sec:diagonal}. The last obstruction is more intricate and relates to the conjugation action on minimal normal subgroups of  $K$, see the notion of conjugation compatibility in \S \ref{sec:conj}.

It is rather easy to deduce from  theorems of Ritt and Burnside that the above obstructions vanish for polynomial compositions $f_1\circ \cdots \circ f_r$ when $\Mon(f_i)$ are nonsolvable, and even under more relaxed conditions, see Corollaries \ref{cor:pols} and \ref{cor:pols iterative}. 
We expect such vanishing would occur much more generally, and in particular for many other rational functions and maps of low genus curves. However, when $\soc(\Gamma)$ is abelian,  the structure of the kernel is much less rigid and we expect in that case more obstructions are needed. From this perspective, the previously studied cases of quadratic and cubic polynomials are perhaps the most subtle.

Finally, we note that the classification of finite simple groups is used through the above only in the proof of Corollary \ref{cor:pols iterative} for the solvability of the outer automorphism groups of simple group, that is, via Schreier's conjecture. This is a nonessential use:   without it, the assertion holds for polynomial maps $f$ whose composition factors admit (nonsolvable) monodromy groups with solvable outer automorphism groups.

\subsubsection*{Further applications}
Theorem \ref{thm:main} is expected to have many further applications, even outside the problems mentioned above. We mention applications to two local-to-global problems for polynomials $f,g$ over the rationals (or over general number fields). 
The first problem, also known as Davenport's problem, asks whether two polynomials $f$ and $g$ which have the same value set modulo $p$ for  all but finitely many primes $p$ (in which case $f$ and $g$ are also called Kronecker-conjugate) are necessarily linearly related. The second problem asks whether, for a polynomial $f$, the \DN{maximal} number $r$ for which the function $\mathbb F_p\ra\mathbb F_p$, $x\mapsto f(x)$ mod $p$ is at least $r$-to-$1$ for all but finitely many primes $p$ coincides with the maximal number $r$ for which  $f:\mQ\ra \mQ$ is at least $r$-to-$1$ over all but finitely many values. While counterexamples for both questions are known, Theorem \ref{thm:main} is used in \cite{Ros} to deduce a positive answer to both problems when no right composition factor of $f$ is linearly related to $x^d$ and $T_d$. 

\subsubsection*{Acknowledgements} 
We thank A.\ Behajaina, R.\ Guralnick, and R.\ Jones for comments and encouragement. J.\ K.\ was supported by the National Research Foundation of Korea (NRF Basic Research Grant RS-2023-00239917). D.\ N.\  was supported by the Israel Science Foundation, grant no.\  353/21,  and is  grateful for the hospitality of the Institute for Advanced Study. J.\ K.\ and D.\ N.\ are grateful for the hospitality of the SQuaREs program of the American Institute of Math. All computer calculations were carried out using MAGMA \cite{Magma}.

\section{Preliminaries}\label{sec:prelim}
\DN{Throughout the paper $F$ is a field of characteristic $0$ and all groups are finite}. 

\subsection{Maps, monodromy and ramification}\label{sec:def}
By a map  $f:X\ra X_0$ over $F$, we mean a \DN{finite (dominant, generically unramified) morphism} of (smooth projective irreducible) \DN{varieties} over $F$. 
Let $\Br(f)\subseteq X_0(\oline F)$ and $\Ram(f)\subseteq X(\oline F)$ denote the branch locus and ramification locus of $f$. 

The monodromy group $\mon_F(f)$ of a map $f$ is the image of the action of the (\'etale) fundamental group $\pi_{1}(X_0\setminus \JK{\Br(f)})$ on the fiber $f^{-1}(x_0)$ of a base point $x_0\in X_0\setminus \JK{\Br(f)}$. \JK{Since $X$ is assumed irreducible, this action is transitive. }
When $F$ is fixed and is clear from the context, we shall simply write $\Mon(f)$ for $\Mon_F(f)$. Note that as an abstract group $G=\Mon(f)$ is the Galois group of a Galois map $\tilde X\ra X_0$ obtained as the quotient corresponding to the kernel of the action of $\pi_1$. 
We shall often identify $f$ with the projection $\tilde X/H\ra \tilde X/G\cong X_0$, where $H\leq G$ is the stabilizer of a point in $f^{-1}(x_0)$.  Moreover since the action of $\Mon(f)$ is equivalent to the action of $G$ on $G/H$, we shall identify the set $\Mon(f)$ acts on with $G/H$. \DN{We say $f$ is Galois if $\Mon(f)$ acts regularly\footnote{Equivalently,  the induced function field extension $F(X)/F(X_0)$ is Galois.}, and that $\tilde f:\tilde X\ra X_0$ is a Galois closure of $f$ if $\tilde f$ is a minimal Galois map such that $\tilde f=h\circ f$ for some map $h$.}

The ramification type $E_f(P)$ of $f$ over $P$ is  defined to be the multiset of ramification indices $e(Q/P)$, where $Q$ runs over preimages $Q\in f^{-1}(P)$. The ramification type of a map $f$ of \DN{curves} is then the (finite) multiset  $E_f(P), P\in \Br(f)$. \DN{For example,  we write the ramification type of the polynomial map given by $x^d$, on affine coordinates, as $[d],[d]$}. The Riemann--Hurwitz contribution over $P$ is defined as $R_f(P)=\sum_{e\in E_f(P)}(e-1)= d - \#f^{-1}(P)$, where $d:=\deg f$. 
For such $f$, the genus $g_X$ of $X$ is given by the Riemann--Hurwitz formula:
$$ 2(g_X-1)=2d (g_{X_0}-1) +  \sum_{P\in \Br(f)}R_f(P). $$
Finally, we recall that 
$E_f(P)$ for $P\in X_0(\oline F)$ coincides with the multiset of orbit lengths of the inertia group $I_P$ over $P$, see e.g.\ \cite[\S 3]{GTZ}. 

\subsection{Imprimitive permutation groups}\label{sec:imprimitive}
Given $U\leq \Sym(I), V\leq \Sym(J)$,  recall that the wreath product $U\wr_J V$ is the semidirect product $\Ind_J U\rtimes V$, where $\Ind_J U$ is the power group $U^J$ equipped with the $V$-action that permutes its coordinates. It is equipped with a natural action on a tree whose branches are indexed by $j\in J$ and whose leaves are indexed by $(i,j)$, $i\in I,j\in J$. The edges connect each $(i,j)\in I\times J$ to $j\in J$. Considering the actions on leaves, we denote the set on which $U\wr_J V$ acts on simply by $I\times J$. 
In this action  $(u_j)_{j\in J}\in U^J$ acts on each branch, that is,  $(u_j)_j\cdot(i,j) = (i,u_j\cdot j)$, and $v\in V$ acts by permuting the branches, that is, $v\cdot (i,j)=(v\cdot i,j)$ for $(i,j)\in I\times J$. We shall denote by $P_J=P_J(I\times J)$ the partition of $I\times J$ whose blocks are the branches $\{(i,j)\in I\times J\,|\,i\in I\}$, $j\in J$. 

Conversely, for a transitive $G\leq \Sym(S)$ and a ($G$-invariant) partition  $J$ of $S$,
let $V$ be the image of the action of $G$ on $J$, and
$U$ the image of the action of the stabilizer of a set $j\in J$ (\DN{a.k.a.\ the block stabilizer}) on $j$. Note that since $G$ is transitive on $J$, the action is isomorphic for all $j\in J$.  It is  well known that then $G$ embeds (as a permutation group) into  $U\wr_J V$. \DN{We call the sets $j\in J$ blocks and say $J$ is a nontrivial partition if $1<\#j<\#S$. Recall that $G$ is imprimitive if there exists a nontrivial partition, otherwise primitive}. 

Finally, we recall the following correspondence between partitions and subgroups:
\begin{lemma}(\cite[Theorem 1.5A]{JhonDixon})\label{lem:correspondence}
Let $G$ act transitively on $S$, and fix $\alpha\in S$. 
Let $G_\alpha\leq G$ be the stabilizer of $\alpha\in S$. Then there is a one to one correspondence 
\begin{equation*}
    \{\text{partitions of }S\}\stackrel{1:1}{\leftrightarrow} \{G_\alpha\leq H\leq G\},
\end{equation*} given by associating to a partition $J$ the stabilizer $G_{J_\alpha}$ of the block  $J_\alpha$ in $J$ containing $\alpha$, and conversely associating to $G_\alpha\leq H\leq G$ the partition 
with block $H\cdot \alpha$.   
\end{lemma}
We note that if furthermore $G=\Mon(f)$, then the two sets also correspond to decompositions  $f=g\circ h$, where $g:\tilde X/G_\alpha\ra \tilde X/H$, $h:\tilde X/H\ra\tilde X/G$ are the natural projections. Up to a \DN{birational equivalence} \DN{of $\tilde X/G_\alpha$}, every decomposition of $f$  is of this form.

\subsection{Kernel properties}\label{sec:ker}
We say that 
a decomposition $f=g\circ h$ is {\it proper} if $\deg g,\deg h>1$.
As above this yields an embedding $\Mon(f)\leq \Mon(h)\wr \Mon(g)$. As in \S \ref{sec: intro - methods}, we say the kernel $K$ of the natural projection $\Mon(f)\ra\Mon(g)$ is {\it large} if $K\leq U^m$ contains $\soc(U)^m$, where $U:=\mon(h)$ and $m:=\deg g$. 

As noted in \S \ref{sec: intro - methods}, we often assume that $\soc(U)$ is a nonabelian minimal normal subgroup of $U$. 
\DN{This is a typical scenario for primitive permutation groups, see the Aschbacher--O'Nan--Scott theorem \cite[Theorem 11.2]{Gur}.}
For such $U$, $\soc(U)$ is a power of a nonabelian simple group, and hence its centralizer $C_U(\soc(U))$ in $U$, and in particular its center $Z(\soc(U))$ is trivial. Largeness ensures:
\begin{lemma}\label{lem:large-min-normal-A}
Let $f=g\circ h$ be a proper decomposition for which the kernel of $\Mon(f)\ra\Mon(g)$ is large. Then the kernel of the projection $\Mon(g_2\circ h)\ra \Mon(g_2)$ is also large for every proper decomposition $g=g_1\circ g_2$. 
\end{lemma}
\begin{proof}
Let $G=\Mon(f)\leq U\wr_J V$, where $U:=\Mon(h)$,  $V:=\Mon(g)$. Let $H\leq H_1$ be stabilizers in the action of $G$ through $V$ and $\Mon(g_1)$, resp. As above identify $J$  with $G/H$. Let  
$G_2:=\Mon(g_2\circ h)\leq U\wr_{J_1} V_2$, where  $V_2 := \Mon(g_2)$ and 
 $J_1:=H_1/H$. 
The group $G_2$  then coincides with the image of the action of $H_1$ on $H_1/H$. In particular, it is the image of $H_1$ under the natural projection  $$ U^J\rtimes \left(\Sym(J_1)\times \Sym(J\setminus J_1)\right)\ra U\wr_{J_1}\Sym(J_1).$$  Since $K=\ker(G\ra V)$ contains $\soc(U)^J$, the image of $K$ under this projection contains $\soc(U)^{J_1}$. Since this image is contained in  $\ker(G_2\ra V_2)$, the claim follows. 
\end{proof}
\begin{lemma}\label{lem:large-min-normal-B}
Suppose $V$ acts faithfully on $J$, and $G\leq U\wr_J V$ projects onto $V$ with large kernel, 
\JK{i.e., $K_0:=\soc(U)^J$} is contained in $K:=\ker(G\ra V)$. 
If $Z(\soc(U))=1$, then $C_G(K_0)=1$ and hence $\soc(G)=\soc(K)=K_0$. 
\end{lemma}
\begin{proof}  
Since  $\soc(U)$ has a trivial center, it is a product of nonabelian simple groups. \DN{As $C_U(\soc(U))$ is a normal subgroup of $U$ which is disjoint from $\soc(U)$, it must be trivial by the definition of $\soc(U)$}. It follows that $K_0$ has a trivial centralizer in  $K\leq U^m$. Since any other minimal normal subgroup of $K$ must centralize $K_0$, we see that $\soc(K)=K_0$. 
To prove that $C_G(K_0)=1$, it remains to note that every lift of a nontrivial element in $V$ to $G$ permutes the coordinates of  $K_0=\soc(U)^{J}$, and hence does not  centralize $K_0$.

To see that $\soc(G)=K_0$, \JK{write $K_0$ as a direct product of nonabelian simple groups. Since $K_0\trianglelefteq G$,  the group $G$ acts by conjugation on the components of this direct product, implying that $K_0$ is a product of minimal normal subgroups of $G$.} Hence $\soc(G)\JK{\supseteq} K_0$. However, as $C_G(K_0)=1$, we deduce $G$ has no other minimal normal subgroups, so that $\soc(G)=K_0$.
\end{proof}

The other extreme of having a large kernel is having a trivial kernel: A proper decomposition $f=g\circ h$ is {\it Galois-proper} if the kernel of the projection $\mon(f)\ra \mon(g)$ is nontrivial. Galois-properness is  clearly  necessary for the kernel of $\Mon(f)\ra \Mon(g)$ to be large. It is also clear that if $f=g\circ h$ is Galois proper, then $g_2\circ h$ is Galois-proper for every proper decomposition $g=g_1\circ g_2$. 
For polynomial maps $f\in F[x]$, proper decompositions $f=g\circ h$ are always Galois-proper by Abhyankar's lemma, see e.g.\  \cite[Lemma 2.8]{KN}.

Finally, we say that $K\leq U^J$ is {\it diagonal} if its projection $K\ra U$
to each of the $J$ coordinates is injective. Note that since $\Mon(g)$ is transitive on $J$, the images of all these projections are isomorphic \cite[Remark 3.2]{KN}. We describe the monodromy groups of decompositions with diagonal kernel in  \S\ref{sec:diagonal}. For a group $U$ with $Z(\soc(U))=1$,  
and an integer $r\ge 1$, we say that an element $x=(x_j)_{\DN{j\in J}} \in U^J$, resp.\ the cyclic subgroup generated by $x$, is {\it diagonal}, if $x_j$, $j\in J$ are all conjugate\footnote{Since $\soc(U)$ has trivial centralizer, we may identify $U$ with a subgroup of $Aut(\soc(U))$.} to each other in $Aut(\soc(U))$.  Otherwise, $x$ is called nondiagonal. In particular, $x$ is nondiagonal as soon as its component entries do not all have the same order.  Clearly, diagonality of elements is preserved under conjugation \JK{in $\Aut(\soc(U))\wr \textrm{Sym}(J)$}. \DN{Diagonality of subgroups and elements are compatible properties, see Lemma \ref{lem:diagvsdiag}.}

\subsection{Specialization}\label{subsec:spec}
To every map $f:X\ra \mP^1_F$, one may associate
the function field extension $F(X)/F(t)$, where $F(t)$ denotes the rational function field. Then, given any $t_0\in \mP^1_F$, the specialization\footnote{The specialization can also be referred to as the ring of coordinates on the fiber of $f$ over $t\mapsto t_0$.}  of $f$ at $t\mapsto t_0$ is the \'etale algebra extension $(\prod_i K_i)/F$, where the $K_i$ are the residue extensions of $F(X)$ at places extending $t\mapsto t_0$. When $f$ is a Galois map, all the $K_i$ are isomorphic, and we identify the specialization with the (Galois) field extension $K_i/F$. We call its Galois group, the Galois group of the specialization of $f$ at $t\mapsto t_0$, as in \cite{DKLN2}. 

When $f$ is given by
a polynomial $P(t,x)\in F(t)[x]$ (by which we mean that the affine curve $X_P$ given by $P(t,x)=0$ is birational to $X$, and furthermore $f$ identifies on a Zariski open subset with the natural projection $X_P\ra \mathbb A^1$, $(t,x)\mapsto t$), there exists a finite set $\Bad_P$ of values $t_0\in \mathbb{P}^1_F$ outside of which specialization of $f$ at $t_0$ is simply given by $P(t_0,x)=0$,\footnote{Concretely  $Bad_P$  consists of the values $t_0$ where $P(t_0,x)$ is undefined, where its $x$-degree is less than that of $P(t,x)$, \DN{and} where $(t-t_0)$ divides the discriminant of $P$ with respect to $x$.} and the Galois group of $P(t_0,x)$ is permutation-isomorphic to a subgroup of the monodromy group $\Mon_F(f)$. 

In the special case $P(t,x)=p(x)-t$ with $p(x)\in F[x]$, corresponding to polynomial maps,  $Bad_P$ is simply the set of critical values of $p$, which is the same as the set of branch points of the underlying map $f$.

We \JK{will make use of the special case of the so-called specialization inertia theorem, which} relates ramification of $f$ with ramification in its specializations, cf.\ \cite{Beck} or \cite[\S 2.2.3]{Leg}. 
For a number field $F$ and a prime ideal $p$ of $O_F$,  two values $a,b \in \mathbb{P}^1(\overline{F})$ are said to meet at $p$ with multiplicity $e>0$, if there is a prime \JK{ideal} $\mathfrak{p}$ of $F(a,b)$ \JK{extending $p$} such that either $\nu_{\mathfrak{p}}(a)$ and $\nu_{\mathfrak{p}}(b)$ are both nonnegative and $\nu_{\mathfrak{p}}(a-b)=e$, or they are both negative and $\nu_{\mathfrak{p}}(1/a-1/b) =e$.
\begin{theorem}
\label{thm:ram_in}
Let $F$ be a number field and $f:X\to \mP^1_F$ a finite Galois 
\DN{map} with branch points $t_1,\ldots,t_r\in \mathbb{P}^1(\overline{F})$ and Galois group $G$. Then there exists a finite set $\mathcal{S}_0$ of prime ideals of $O_F$ such that for all $t_0\in \mathbb{P}^1_F$ and for all prime ideals $p\notin \mathcal{S}_0$ of $O_F$, the following holds. If there exists a branch point $t_i$ of $f$ such that $t_0$ and $t_i$ meet at $p$ with multiplicity $1$, then $p$ ramifies in the specialization at $t_0$, with inertia group conjugate to $\langle\tau\rangle$, where $\tau$ is the inertia group generator at $t\mapsto t_i$. 

If additionally $G$ has a trivial center \footnote{The trivial center assumption only ensures that in \cite[\S 2.2.3]{Leg} there are no additional exceptional primes arising from ``vertical ramification", see \cite[Proposition 2.3]{Beck}.}, then one may concretely choose $\mathcal{S}_0$ as the union of 
the set of primes dividing $|G|$;
the finite set of primes at which two branch points meet (with multiplicity $\geq 1$); and the set of primes ramifying in the extension of $F$ generated by all the branch points. 
\end{theorem}

\section{Properties of decompositions that are necessary for largeness}\label{sec:nec}
\subsection{Invariant decompositions}\label{sec:inv} 
The following proposition shows that the largeness of the kernel $\Mon(f)\ra \Mon(g)$ also implies an {\it invariance} property of a proper decomposition $f=g\circ h$: 
Say that $f=g\circ h$ is {\it invariant} if for every  decomposition $f= g_1 \circ  h_1$, the maps $h$ and $h_1$ have a nontrivial common right composition factor $w$, that is,   $ h = u \circ w$ and $h_1 = v \circ w$,  for some $w$ of degree $>1$. 
For  indecomposable $h$, the invariance of $f=g\circ h$ amounts to $f=g\circ h$ being {\it right-unique}  \cite{Ros}, that is,  $h$ is right factor of $h_2$ for every proper decomposition $f=g_2\circ h_2$. Clearly, if $f=g\circ h$ is invariant, then also $g_2\circ h$ for every proper decomposition $g=g_1\circ g_2$. 

\begin{prop}\label{lem:inv-nec}    
Let $f=g\circ h$ be a proper decomposition such that 
$Z(\soc(U))=1$ for 
$U:= \Mon(h)$. If $\ker(\Mon(f)\ra \mon(g))$ is large, then $f=g\circ h$ is invariant.  
\end{prop}
\begin{proof}
Suppose on the contrary that there exists another decomposition $f=\hat g_2\circ h_2$ such that $h$ and $h_2$ have no nontrivial common right composition factor. 
Let $G:=\Mon(f)$, let $\tilde X\ra\tilde X/G$ be the Galois closure of $f$, and $G_1\le G$ the stabilizer of a point $1$.
Let $G_1\leq H, H_2\leq G$ denote the stabilizers of blocks containing $1$ in the actions through $\Mon(g), \Mon(\hat g_2)$, respectively. Since $h$ and $h_2$ have no nontrivial common right composition factor, we have $H\cap H_2 = G_1$ \DN{in view of the correspondence in \S\ref{sec:imprimitive}}. 

We first replace the decompositions $f=g\circ h=\hat g_2\circ h_2$ with the decompositions $f_1:=g_1\circ h=g_2\circ h_2$, where $g_1,g_2$ correspond to the projections $\tilde X/H\ra\tilde X/\langle H,H_2\rangle$ and $\tilde X/H_2\ra\tilde X/\langle H,H_2\rangle$, respectively.  Let $\oline G:= \Mon(f_1)$, i.e., $\oline G$ is the image of $\langle H,H_2\rangle$ in the action on $\langle H,H_2\rangle/G_1$. Let $\oline H,\oline H_2,\oline G_1$ be the images in $\oline G$ of $H,H_2,G_1$, respectively. 
Then $V=\Mon(g_1)$ acts transitively on the blocks $J:=\oline G/\oline H$. Since the kernel of the projection $\oline G\ra V$ (that is, $\Mon(f_1) \ra \Mon(g_1)$) is also large by Lemma \ref{lem:large-min-normal-A}, we have  $\soc(U)^J\leq \oline G\leq U\wr_J V$. 

Since $\overline{H}$ is the stabilizer of the block $J_1\in J$ containing $1$, it also contains $\soc(U)^J$. Moreover $\soc(U)^{J\setminus\{J_1\}}$ is contained in the kernel of the action on $\oline H/\oline G_1$, and hence also in $\oline G_1$ and $\oline H_2$. On the other hand since $\oline G$ acts faithfully by definition, the core of $\oline G_1$ in $\oline G$ is trivial, and hence  
$\soc(U)^J$ and the $J_1$-th copy of $\soc(U)$ are not contained in $\oline G_1$. Since these group are contained in $\oline H$, and $\oline H_2\cap \oline H=\oline G_1$, the groups are also not contained in $\oline H_2$.  
 It follows that the action of $\oline H_2$ on $J$, via its conjugation action on the $\#J$ copies of $\soc(U)$, fixes $J_1$ (setwise). Since $\oline H_2$ and $\oline H$ both fix $J_1$, it follows that $\oline G=\langle\oline H,\oline H_2\rangle$ fixes $J_1$, contradicting the transitivity of $V$ on $J$. 
\end{proof}

\begin{remark}\label{rem:invariance}
    The proof furthermore shows that for every proper intermediate subgroup $G_1\lneq \hat H\leq G$ the intersection  $\hat H\cap H$ properly contains $G_1$. For indecomposable $H$, this amounts to the assertion that every proper intermediate subgroup $G_1\lneq \hat H\leq G$ contains $H$. In view of the correspondence in Lemma \ref{lem:correspondence}, this implies that $J$ is finer than any nontrivial partition of $G/G_1$. 
\end{remark}

Invariance is also part of the sufficient condition for having a large kernel. Its effect on kernels is explained by the following direct consequence from \cite[\S 3]{KN}: 

\begin{prop}\label{prop:carmel}
Let \JK{$V$ act transitively on $J$, and let } $G\leq U\wr_J V$ be a subgroup that projects onto $V$ with nontrivial kernel $K\leq U^J$, such that the block stabilizer of $G$ projects onto $U$, and  $\soc(U)$ is a nonabelian minimal normal subgroup of $U$. Write  $\soc(U)=L^I$ for a nonabelian simple group $L$.  Then $G$ acts transitively on a partition $P=\{P_1,\ldots,P_r\}$ of  $I\times J$ such that $K\cap L^{P_i}\cong L$, $i=1,\ldots r$,   and 
$\soc(K)$ is a minimal normal subgroup of $G$ fulfilling $\soc(K) = K\cap \soc(U)^J = (K\cap L^{P_1})\times \cdots\times (K\cap L^{P_r})$.

If furthermore $U=\Mon(h)$, $V=\Mon(g)$, and $G=\Mon(f)$ for an invariant decomposition $f=g\circ h$, 
then  $\soc(G) = \soc(K)$ is minimal normal in $G$.
\end{prop}
\begin{proof}
The first assertion is \cite[Lemma 3.1]{KN}, together with \cite[Corollary 3.4]{KN} for the minimality of $\soc(K)$. For the ``furthermore" assertion, let $G_1$ be a point stabilizer and  $G_0$ the stabilizer of the corresponding block, so that $G_1\leq G_0$. 
Let $\tilde f:\tilde X\ra Y$ be the Galois closure of $f$, and assume on the contrary there is another minimal normal subgroup $1\neq N\lhd G$. Then \cite[Lemma 3.5]{KN} yields a subgroup $G_1N\supsetneq G_1$ whose intersection with $G_0$ is $G_1$. 
Identifying $X\cong \tilde X/G_1$, we see that $f$ factors as $g_1\circ g_2$, where $g_2$ is the natural projection $\tilde X/G_1\ra \tilde X/G_1N$. 
Note that since $\core_G(G_1N)\supseteq N$, $f=g_1\circ g_2$ is a proper decomposition. 
Since the intersection of $G_1N$ and $G_0$ is $G_1$, it follows that $g_2$ and $h$ have no nontrivial common right composition factor, contradicting the invariance of the decomposition.  
\end{proof}

\begin{remark}
    \label{rem:burnside}
For indecomposable $h\in F[x]$, $d:=\deg h\geq 5$, that is not linearly related over $\oline F$ to $x^d$ or $T_d$,  $\Mon(h)$ is a nonabelian almost simple group by theorems of Schur and Burnside, see e.g.\ Theorem \cite[Theorem 2.1]{KN} or \cite{Mul}. Thus, in Proposition \ref{prop:carmel}, $\#I=1$  and the proposition yields a partition  of $J$ itself. 
\end{remark}

\subsection{\JK{Nondiagonality}}\label{sec:diagonal}

When the decomposition $f=g\circ h$ is invariant and Galois-proper, the following lemma describes   $\mon(f)$ when the kernel is diagonal:
\begin{lemma}\label{lem:diag}
Let $f=g\circ h$ be a Galois-proper invariant decomposition  with monodromy groups $G:=\mon(f)\leq U\wr V$, $U:=\mon(h)$, and $V:=\Mon(g)$ of degree $m$. Assume that $\soc(U)$ is a nonabelian minimal normal subgroup of $U$, \JK{ and let $K:=\ker(G\ra V)$}. 
If $\soc(K)\leq U^m$ is a diagonal subgroup,
then  $G$ embeds into $\Aut(\soc(U))$. 
\end{lemma}
\begin{proof}
First, let $\rho:G\ra \soc(K)$ denote the conjugation action and  $C:=\ker \rho$ its kernel which is no other than the centralizer $C_G(\soc(K))$ of $\soc(K)$. Note that  $K\neq 1$ since $f=g\circ h$ is Galois proper,  and hence $\soc(K)\neq 1$. 

Since  $\soc(U)$ is a nonabelian minimal normal subgroup of $U$, it is a power of a nonabelian simple group $L$. Since in addition $K\neq 1$,    Proposition \ref{prop:carmel} implies that  $\soc(K)$ is also a power of $L$. Thus, these groups have a trivial center and hence $C\cap \soc(K)=1$. Moreover, since $f=g\circ h$ is invariant, Proposition \ref{prop:carmel} implies that $\soc(K)$ is the unique minimal normal subgroup of $G$, and hence $C=1$ and $\rho$ is injective. 

Since $\soc(K)$ is diagonal, we may identify $\soc(K)$ with $\soc(U)$ via projection to a coordinate. Thus  $\rho$ yields  an embedding  
$G\ra \Aut(\soc(K))=\Aut(\soc(U)).$
\end{proof}
Note that the embedding $\hat\rho:G\ra\Aut(\soc(U))$ is given explicitly by  combining the conjugation action $\rho:G\ra \Aut(\soc(K))$ with the identification $\soc(K)\ra\soc(U)$.

Conversely, the existence of such an embedding yields diagonality:
\begin{lemma}
\label{lem:diag-nec}
Let $f=g\circ h$ be a proper decomposition, and assume  $\soc(U)$ is a nonabelian minimal normal subgroup of $U:=\Mon(h)$. Let $K:=\ker(\Mon(f)\ra \Mon(g))$. If there exists an embedding $\hat \rho:G\ra\Aut(\soc(U))$, then $\soc(K)$ is a diagonal subgroup.
\end{lemma}

\begin{proof}
Since $\soc(U)$ is a nonabelian minimal normal subgroup, it is of the form  $L^r$, $r\geq 1$ for a nonabelian simple group $L$.
Since we may in addition assume $K\neq 1$ and hence $\soc(K)\neq 1$,  Proposition \ref{prop:carmel} implies that $\soc(K)$ is a direct power of $L$, say $\soc(K)\cong L_1\times \dots\times L_m$ with some $m\in \mathbb{N}$ and $L_i\cong L$ for all $1\le i\le m$. 

We claim that $m\le r$. Since  $\soc(K)\neq 1$, the component projections $\soc(K)\ra \soc(U)$ are nontrivial 
and hence their images contain the minimal normal subgroup $\soc(U)= L^r$ of $U$, so that the claim yields
$m=r$ and that the component projections are injective, as required.  

To show the claim, consider the embedding $\hat\rho: \soc(K)\to \Aut(\soc(U))$. Note that $\Aut(\soc(U))=\Aut(L^r)$ acts on the $r$ components of $L^r$, with kernel $N=Aut(L)^r$. Let $J$ be the set of indices $j\in \{1,\dots, m\}$ for which $\hat\rho(L_j)\subseteq N$. Since $\Out(L)$ is solvable by Schreier's conjecture, one then even has $\hat\rho(L_j)\subseteq L^r$. Now, for each $j\in J$, let $I_j\subseteq\{1,\dots, r\}$ be the set of components onto which $\hat\rho(L_j)$ projects nontrivially (and thus, automatically surjects, since $L$ is simple). Since the $L_j$, $j\in J$ commute pairwise, the sets $I_j$, $j\in J$ must be pairwise disjoint.  

Next, note that $\hat\rho(\prod_{i\in \{1,\dots, m\}\setminus J} L_i)$ acts faithfully on the components by the definition of $J$ \DN{and since $L$ is nonabelian simple}. This image is a subgroup of $S_r$, and contains an abelian subgroup $A$ of order at least $2^{m-\#J}$. Let $O_1,\dots, O_k\subseteq \{1,\dots, r\}$ be the orbits of $A$. Since $A$ is abelian, the faithfulness of the action enforces $\prod_{i=1}^k\# O_{\DN{i}} \ge |A|\ge 2^{m-\#J}$. Taking a logarithm, one obtains $\sum_{i=1}^k (\#O_i-1) =r-k \ge m-\#J$. 
But on the other hand, $\prod_{i\in \{1,\dots, m\}\setminus J} L_i$ commutes with each $L_j$, $j\in J$, which forces its image to fix each $I_j$ setwise. Thus, each of the disjoint sets $I_j$ is a union of certain orbits $O_i$, $1\le i\le k$, and in particular $\#J\le k$. The combination of the above inequalities yields $m\le r-k+\#J\le r$, showing the claim. 
\end{proof}

\begin{remark}\label{remark:polynomial monodromy do not embeed into automorphism group}
The condition that $G=\Mon(g\circ h)$ does not embed into $Aut(\soc(U))$ can also often be easily verified very conveniently regardless of diagonality considerations as above. Notably, if $\soc(U)$ is nonabelian simple for $U=\Mon(h)$, and $\Mon(g)$ is nonsolvable, then such an embedding is impossible since $\Out(\soc(U))$ is solvable by Schreier's conjecture. 

For polynomial maps $h\in F[x]$, the last condition on $U$ is satisfied as in Remark \ref{rem:burnside}, as long as  $h$ is not linearly related over $\bar{F}$ to $x^d$ or $T_d$, and $\deg h\ge 5$. If $g$ is also a polynomial map of degree $>1$, then the extra condition on $\Mon(g)$ can even be dropped. Indeed, an embedding of $\Mon(g\circ h)$ into $\Aut(\soc(\Mon(h)))$ would yield a cyclic transitive subgroup in $\Mon(h)$ (the inertia group at infinity) extended by a cyclic subgroup of strictly larger order in $\Aut(\soc(\Mon(h))$. This, however, does not happen as a consequence of the classification of finite simple groups, see \cite[Theorem 1.2(2)]{LiPr}.\end{remark}
%

\subsection{Conjugation compatibility}\label{sec:conj} The following lemma describes a property which  is necessary for the largeness of kernels by Lemma \ref{lem:full}. 
\begin{lemma}\label{lem:partition}
Suppose $G$ acts on $S$ transitively with two nontrivial partitions $P$ and $Q$.
Let $G_p,G_q$ denote stabilizers of blocks $p\in P$, $q\in Q$, resp., 
and $\pi:S\ra Q$ the  natural map. 
Then the following conditions are equivalent:
\begin{enumerate}
\item The sets $\pi(p)$, $p\in P$ form a  partition of $Q$;
\item The following is an equivalence relation on $S$:  \\
\centerline{$x\sim y$ if $\pi(x)$ and $\pi(y)$ belong to $\pi(p)$ for some $p\in P$;}
\item $G_pG_q=G_qG_p$ for every $p\in P$, $q\in Q$ with $p\cap q\neq\emptyset$;
\item $G_p\cdot G_q$ is a group for every  $p\in P$, $q\in Q$ with $p\cap q\neq \emptyset$.
\end{enumerate}
\end{lemma} 
\begin{proof}
The equivalence  (1)$\Leftrightarrow$(2) is immediate from the definitions of $\pi$ and the relation. 
The equivalence of (3)$\Leftrightarrow$(4) is a well known elementary fact in group theory. 
To show the equivalence  (1)$\Leftrightarrow$(3), fix $q\in Q$ and note that the action of $G$ on $Q$ is equivalent to its action on $G/G_q$ via the mapping $x\cdot q\mapsto xG_q$, $x\in G$. Thus, the equality $\pi(p)=\pi(p')$ is equivalent to $xG_pG_q=G_pG_q$ for  $x\in G$ such that $xp=p'$. 

Since $G_q$ is transitive on $q$, for every $p,p'\in P$ which intersect $q$ nontrivially there is an $x\in G_q$ such that $xp=p'$. Thus, Condition (1) is equivalent to the condition: 
\begin{equation}\label{eq:coset-cond}
\tag{1'}
    xG_pG_q=G_pG_q\text{ for every $x\in G_q$, $p\in P$,$q\in Q$ with $p\cap q\neq\emptyset.$}
\end{equation}
The condition clearly holds if (3) holds. Conversely, if \eqref{eq:coset-cond} holds, then $xy\in G_pG_q$ for every $x\in G_q,y\in G_p$ so that $G_qG_p\subseteq G_pG_q$. Since $G$ is finite this implies equality. 
\end{proof}
To assert that the equivalent conditions of Lemma \ref{lem:partition} hold, we shall say for short that the partitions $P$ and $Q$ are {\it compatible}. 
The necessary condition for largeness is: 
\begin{defn}
Let  $f=g\circ h$ be a decomposition such $\soc(U)=L^I$ is a nonabelian minimal normal subgroup of $U:=\Mon(h)$, and $V:=\mon(g)$ acts on $J$, so that $G=\Mon(f)\leq U\wr_J V$ acts on $I\times J$ via conjugation. Let $P$ be the partition of $I\times J$ from Proposition \ref{prop:carmel}, and $P_J$ the partition of $I\times J$ induced by $J$. 
We say that $f=g\circ h$ is a {\it conjugation-compatible decomposition} if $P$ and $P_J$ are compatible. 
\end{defn}

Clearly if one of the partitions $P_J$ or $P$ is a trivial partition then $P$ and $J$ are compatible and the decomposition is conjugation compatible. Thus, if the kernel in Proposition \ref{prop:carmel} contains $L^{I\times J}$, the decomposition is conjugation compatible:
\begin{lemma}\label{lem:full}
Suppose $f=g\circ h$ so that $G=\Mon(f)$ is a subgroup of $U\wr_J V$, where $U=\Mon(h)$, and $V=\Mon(g)$ acts faithfully on $J$. Suppose $\soc(U)=L^I$ is a nonabelian minimal normal subgroup of $U$ and that the kernel of the projection $G\ra V$ is large.  
Then $f=g\circ h$ is a  conjugation-compatible decomposition. 
\end{lemma}
\begin{proof}
    Since $G$ contains $L^{I\times J}$, the corresponding partition $P$ in Proposition \ref{prop:carmel} is the trivial  partition of $I\times J$ into  sets $\{(i,j)\}$, $i\in I,j\in J$. 
    In particular, $P$ and the partition $P_J$ induced by $J$ are compatible and $g\circ h$ is conjugation compatible. 
\end{proof}


The following is an example of a decomposition that is not conjugation compatible:
\begin{example}
Let $\tilde f:\tilde X\ra X_0$ be a Galois  map  with group $G=A_n^3\rtimes S_3=\prod_{i=1}^3A_n^{\{i\}}\rtimes S_3$ for $n\geq 5$. We shall produce a decomposable map $f=g\circ h$ with Galois closure $
\tilde{f}$, for natural projections $g:\tilde X/G_0\ra X_0$ and $h:\tilde X/G_1\ra \tilde X/G_0$ for certain subgroups $G_1\leq G_0\leq G$, and show it is not conjugation compatible. 

To produce $G_0,G_1$, let $S_3$ act on the set $J$ of $2$-subsets of $\{1,2,3\}$. Let $G_0=A_n^J\rtimes \langle s\rangle$, for $s=(1,2)$, be a block stabilizer (of $\{1,2\}\in J$), and let $g:\tilde X/G_0\ra X_0$  be the corresponding projection. Consider the product action of $G_0$ on $\{1,\ldots,n\}^{\{1,2\}}$, where $s$ acts by permuting the two coordinates and $A_n^{\{1\}}\times A_n^{\{2\}}$ acts pointwise. The kernel of this action is the third copy $A_n^{\{3\}}$  and its image is the primitive wreath product $U=A_n\wr \langle s\rangle$ with socle $A_n^I$, $I:=\{+,-\}$. Let $G_1\leq G_0$ be a point stabilizer in this action of $G_0$ and $h:\tilde X/G_1\ra\tilde X/G_0$ the corresponding projection. 

The resulting map $f=g\circ h$ has Galois closure $\tilde f$ inducing an injection $G\ra U\wr S_J$, as in the beginning of the section. The conjugation action on $\soc(U)^J=A_n^{I\times J}$ induces an action of $G$ on $I\times J$ with block kernel $K\cong A_n^3$. This action factors through the projection $G\ra S_3$ which acts faithfully on $I\times J$. Since in addition $\#(I\times J)=6$, $S_3$ acts regularly on $I\times J$. In this action, the stabilizer of a block $I\times \{j\}$, \DN{$j:=\{1,2\}$}, in the partition $P_J$ corresponding to $J$ is $K\langle (1,2)\rangle$. On the other hand, the partition $P$ of $I\times J$ given by Proposition \ref{prop:carmel} has  $3$ blocks and  of cardinality $2$. Thus, the stabilizer $H\leq S_3$ of a block in $P$, say of the block containing $(+,j)$   is a transposition $\sigma$. We claim that $\sigma\neq (1,2)$ and hence $\langle \sigma\rangle \cdot \langle(1,2)\rangle$ is not a group, so that $P$ and  $P_J$ are not compatible by Lemma \ref{lem:partition} and hence $f=g\circ h$ is not conjugation compatible. 

It remains to show the last claim: If on the contrary $\sigma=(1,2)$ then, by definition of $P$, the first copy $A_n^{\{1\}}\leq A_n^3$ is mapped to a diagonal subgroup 
of $A_n^{I\times\{j\}}\leq A_n^{I\times J}$ which is not normal in $\soc(U)=A_n^I$. However, the projection of $\soc(G)=A_n^3$ to the $(I\times\{j\})$-th block in $A_n^{I\times J}$ is all of $\soc(U)=A_n^I$ by construction. This contradicts the normality of $A_n^{\{1\}}$ in $A_n^3$, and proves the claim. 
\end{example}

\section{Main theorem and consequences}
\label{sec:main}
As before fix $F$ to be a field of characteristic $0$, and assume all maps $f:X\ra X_0$ are defined over $F$. Our main theorem is: 
\begin{theorem}\label{thm:main}
Let $f=g\circ h$ be an invariant conjugation-compatible decomposition. Assume $\soc(U)$  is a nonabelian minimal normal subgroup of $U:=\Mon(h)$. 
For every decomposition $g=g_1\circ g_2$,  $\deg g_2>1$, assume 
$g_2\circ h$ is Galois-proper and $\Mon(g_2\circ h)$ does not embed in $\Aut(\soc(U))$. 
Then the kernel of $\mon(f)\ra \mon(g)$ is large. 
\end{theorem}
For maps $h$ with monodromy group $U$ whose socle is nonabelian minimal normal, 
the rest of the assumptions are all necessary:
\begin{remark}\label{rem:nec}
    1) The invariance condition on $f=g\circ h$ is necessary by Proposition \ref{lem:inv-nec}. \\ 
    2) The conjugation-compatibility condition is necessary by Lemma \ref{lem:full}. It holds trivially if $U$ is almost simple, see \S\ref{sec:conj}.  \\
    3) The assumption that $g_2\circ h$ is Galois-proper is necessary and holds for polynomial maps, see \S\ref{sec:ker}. If it fails for some $g_2$, it is sometimes still possible to apply the theorem with $h$ replaced by $g_2\circ h$. \\
    4) The condition that $\Mon(g_2\circ h)$ does not embed into $\Aut(\soc(U))$ is also necessary, see Lemma \ref{lem:diag-nec}. It holds, \JK{e.g.,} for polynomial maps, 
    see Remark \ref{remark:polynomial monodromy do not embeed into automorphism group}. As the following example shows, to make sure  the condition holds, it is sometimes necessary to replace $h$ by $g_2\circ h$. 
\end{remark}
\begin{example}
    Let $f=g_1\circ g_2\circ h$, $\deg g_2>1$, and assume $\soc(U)$ is a nonabelian minimal normal subgroup of $U:=\Mon(h)$. If $U'=\Mon(g_2\circ h)\leq \Aut(\soc(U))$, then $G:=\Mon(f)$ is a subgroup of $(U')^{\deg g_1}\rtimes V$ for $V=\Mon(g_1)$. In particular, the kernel of the projection $G\ra V$ may only contain $\soc(U')^{\deg g_1}=\soc(U)^{\deg g_1}$ but not the larger group $\soc(U)^{\deg g}$ where $g=g_1\circ g_2$. In this setup,  the map $h$ from Theorem \ref{thm:main}  should be picked to be the map $g_2\circ h$. 
\end{example}
For the reader interested mainly in the case of polynomials, we recommend reading the following proof under the simplifying assumption that $U$ is almost simple. 
\begin{proof}[Proof of Theorem \ref{thm:main}]
Let $V=\Mon(g)$ act on a set $J$, so that $G:=\Mon(f)$ is identified with a subgroup of $U\wr_J V$. 
Since $\mon(g\circ h)$ is larger than $\mon(g)$, the  projection $G\ra V$ has a nontrivial kernel $K\neq 1$. Since in addition $\soc(U)$ is the nonabelian  minimal normal subgroup of $U$,  it is a direct power $L^{I}$ of a nonabelian simple group $L$. Proposition  \ref{prop:carmel} then implies that $\soc(G)=\soc(K)=  \prod_{p\in P} (L^{p}\cap K)\cong L^P$ for a partition $P$ of $I\times J$. Since the decomposition is conjugation compatible,  $P$ is compatible with the partition $I_j:=I\times\{j\},j\in J$, induced by $f=g\circ h$,  and hence the projection $\pi:I\times J\ra J$ induces a partition $\pi(P)$ of $J$ by Lemma \ref{lem:partition}.

Pick $p \in P$ and let $q:=\pi(p)\subseteq J$ be a block in $\pi(P)$. Let $I_q:=\pi^{-1}(q)\subseteq I\times J$ be the union of blocks $I_j$, $j\in J$. Identifying via the above isomorphism the direct power $\soc(U)^q$ of $\#q$ copies of $\soc(U)$ with the direct power $L^{I_q}$ of $\#q$ copies of $L^I$, we claim that  $\soc(U)^q\cap K$ is a diagonal subgroup of $\soc(U)^q$: Pick $j \in q$, and note  that $j\in \pi(p')$ for $p'\in P$ if and only if $q=\pi(p')$ since $\pi(P)$ is a partition. Since the $j$-th projection maps $\soc(K)=\soc(U)^J\cap K$ onto $\soc(U)$, it maps the subgroup $\soc(U)^q\cap K=\prod_{p'\in P: \pi(p')=q} (L^{p'} \cap K)$ of $\soc(K)$ onto $\soc(U)$. Moreover, the $j$-th projection from $\prod_{p': \pi(p')=q} (L^{p'} \cap K)$ to $\soc(U)$ is an isomorphism since both sides are powers of $L$ of the same cardinality: for,  $\soc(U)=L^I$,  each direct factor   $L^{p'} \cap K$ in the product is isomorphic to $L$ by definition of $P$, and $I_q$ consist of the $\#I$ blocks $p'\in P$ satisfying $\pi(p')=q$. The resulting isomorphism yields the claim. 

Letting $\tilde X\ra \tilde X/G$, denote the Galois closure of $f$, we identify  $g$ with the projection $\tilde X/G_j \ra \tilde X/G$, where $G_j$ is the stabilizer of $j\in J$. Furthermore, write $g=g_1\circ g_2$, where $g_2:\tilde X/G_j\ra \tilde X/G_q$ is the quotient map by the stabilizer $G_q$ of a block $q\in \pi(P)$ containing $j$. 
Then $\Gamma_2:=\mon(g_2\circ h)$ is the image of the action of $G_q$ on the union $Q=\bigcup_{j\in q} j$ of blocks  in $q$. 

We next claim that the image of $\soc(K)$ in $\Gamma_2$ is the socle $\soc(K_2)$ of  $K_2:=\ker(\Gamma_2\ra \Mon(g_2))$.
Set $m_2:=\deg g_2=[G_q:G_j]=\#q$.
Since the decomposition $g_2\circ h$ is Galois-proper and invariant, and since $\soc(U)$ is the nonabelian minimal normal subgroup of $U=\Mon(h)$, $\soc(K_2)=K_2\cap \soc(U)^{m_2}$  is the unique minimal normal subgroup of $\Gamma_2$ by Proposition \ref{prop:carmel}. 
Since $\soc(K)\lhd G_q$ acts nontrivially on $Q$ and trivially on (the set of block $j$ in) $q$, its projection to $\Gamma_2$ is a nontrivial  subgroup of $K_2$ which is normal in $\Gamma_2$, and hence contains $\soc(K_2)$,  \DN{by the unique minimality of the latter. On the other hand, the projection of $\soc(K)\leq \soc(U)^J$ is a subgroup of $\soc(U)^{m_2}$ and hence of $K_2\cap \soc(U)^{m_2}=\soc(K_2)$},  
yielding the equality.  

Since $\soc(U)^q\cap K$ is diagonal in $\soc(U)^{q}$ by the above claim, it follows that its projection $\soc(K_2)$ to $\Gamma_2$ is diagonal in $U^{m_2}$.
If $\pi(P)$ is strictly coarser than $J$, then $m_2>1$, and Lemma \ref{lem:diag} implies that $\Gamma_2=\mon(g_2\circ h)$ is a subgroup of $\Aut(\soc(U))$ contradicting our assumption. Thus, $\pi(P)$ is the trivial partition of $J$ and hence $G$ contains $L^{I\times J}=\soc(U)^{J}$. 
\end{proof}
Note that the proof uses properness, invariance, and the condition of non-embeddability in $\Aut(\soc(U))$  only for  specific decompositions  $g_2\circ h$, where $g=g_1\circ g_2$.

For polynomial maps the assumptions of Theorem \ref{thm:main}  simplify as follows. Recall from \S\ref{sec:inv} that for indecomposable polynomial maps $h$ the invariance of a decomposition $g\circ h$ coincides with its right-uniqueness.  

\begin{cor}\label{cor:pols}
Let $f=g \circ h$ be a right unique decomposition with  $g,h \in F[x]$ and $h$ indecomposable of degree $d\geq 5$ that is not linearly related to $x^d$ or $T_d$ over $\oline F$. 
Then the kernel of the projection from $\Mon(f)$ to $\Mon(g)$ is large. 
\end{cor}
\begin{proof}
Since $d\geq 5$ and $h$ is not $x^d$ or $T_d$ up to composition with linear polynomials, as in Remark \ref{rem:burnside}, 
$U:=\Mon(h)$ is a nonabelian almost simple group. In particular, $\soc(U)$ is nonabelian minimal normal subgroup of $U$.
Since $\soc(U)$ is simple, the decomposition $f=g \circ h$ is conjugation compatible as noted in \S\ref{sec:conj}. Moreover, letting $g = g_1 \circ g_2$ be a nontrivial decomposition, the decomposition $g_2\circ h$ is invariant since $g\circ h$ is  right-unique. As $g_2$ and $h$ are polynomials, $g_2 \circ h$ is Galois-proper, as noted in \S\ref{sec:ker}. Moreover, $\mon(g_2\circ h)$ does not embed into $\Aut(\soc(U))$ by Remark~\ref{remark:polynomial monodromy do not embeed into automorphism group}. 
Thus, the conditions of Theorem \ref{thm:main} hold and it implies that $K$ contains $\soc(U)^J$. 
\end{proof}

Applying Corollary~\ref{cor:pols} iteratively yields the following.
\begin{cor}\label{cor:pols iterative}
Let $f=f_1\circ \cdots \circ f_r$ for indecomposable polynomials $f_1,\ldots,f_r\in F[x]$ of degree at least $5$ that are not linearly related to $x^d$ or $T_d$ over $\oline F$ for any $d\in \mathbb{N}$. Then $\Mon(f)$ contains the composition factors  of  $[\soc(\Gamma_i)]_{i=1}^r$, where $\Gamma_i:=\Mon(f_i)$. 
\end{cor}
\begin{proof}
Let $K$ be the kernel of the projection from $\Mon(f)$ to $\Mon(g)$ for $g:=f_1\circ\cdots\circ f_{r-1}$. 
Since the decomposition $f=g \circ f_r$ is right unique by \DN{Ritt's theorems \cite[Thm.\ 2.1 and Thm.\ 2.17]{MZ}}, the conditions of Corollary \ref{cor:pols} hold and the corollary implies that $K$ contains $\soc(\Gamma_r)^{\deg g}$. 
Arguing inductively that $G/K = \Mon(g)$ contains the composition factors of $[\soc(\Gamma_{i})]_{i=1}^{r-1}$, we get that $G$ contains all those of $[\soc(\Gamma_{i})]_{i=1}^{r}$.
\end{proof}
To deduce Theorem \ref{thm:main-intro}, it remains to show that a subgroup containing the composition factors of $[A_{n_i}]_{i=1}^r$ contains a copy of it, see Proposition \ref{prop:splitting}. Corollary \ref{cor:spec}, however, follows immediately from Corollary \ref{cor:pols} and Hilbert's irreducibility theorem. 

It is also possible to deduce Corollary \ref{cor:pols iterative} from the following useful consequence of the proof of Theorem \ref{thm:main}. 
\begin{cor}\label{lem:consecutive}
Suppose $f=f_1\circ \cdots\circ f_r$ for indecomposable maps $f_i$ of degree $d_i$ whose monodromy groups $\Gamma_i$ are nonabelian almost-simple for $i=1,\ldots,r$, and write $g_i:=f_1\circ \cdots \circ f_i$. Assume that the Galois closures of $g_i$, $i=1,\ldots,r$ are all distinct and that the kernels $K_i\leq \Gamma_{i}^{d_i}$ of the projections $\Mon(f_{i-1}\circ f_{i})\ra \Gamma_{i-1}$, $i=2,\ldots,r$ are nondiagonal. Then $\Mon(f)$ contains the composition factors of $[\soc(\Gamma_i)]_{i=1}^r$. 
\end{cor}
\begin{proof}
Let $G:=\Mon(f)$ and $\tilde X\ra\tilde X/G$ the Galois closure of $f$. 
Let $H_r\leq H_{r-1}\leq \ldots\leq H_0=G$ be the chain of point stabilizers corresponding to $f_1\circ \cdots \circ f_r$, so that we identify $f_i$ with $\tilde X/H_i\ra\tilde X/H_{i-1}$ for $i=1,\ldots,r$. Since the maps are indecomposable, $H_{i+1}$ is maximal in $H_i$, $i=0,\ldots,r-1$.  

We argue by induction on $r$ that $\Mon(f)$ contains $[\soc(\Gamma_i)]_{i=1}^r$ and furthermore $H_r\leq \ldots\leq H_0$ is the unique chain of maximal subgroups between $H_r$ and $G$, so that  $f=f_1\circ \cdots \circ f_r$ is the unique decomposition of $f$ up \DN{to birational equivalence,} as in \S\ref{sec:imprimitive}.
The base case $r=1$ holds trivially. So assume $r\geq 2$, and inductively that $H_{r-1}\leq \cdots \leq H_0=G$ is the unique chain of maximal subgroups between $H_{r-1}$ and $G$ and that $\Mon(g_{r-1})$ contains the composition factors of $[\soc(\Gamma_i)]_{i=1}^{r-1}$. 

We claim that the socle $\soc(K)$ of  $K=\ker(\Mon(f)\ra\Mon(g_{r-1}))$ is large. This claim clearly implies that $G$ contains all composition factors of $[\soc(\Gamma_i)]_{i=1}^r$, and by Proposition \ref{lem:inv-nec} and Remark \ref{rem:invariance}.1, that every proper intermediate subgroup $H_r\lneq H\lneq G$ contains $H_{r-1}$. Combining with the induction hypothesis, we deduce that $H_r\leq \ldots\leq H_0=G$ is the unique chain of maximal subgroup between $H_r$ and $G$, completing the induction and the proof. 


To prove the claim note that by Proposition \ref{prop:carmel},  $K=L^P$ where $L=\soc(\Gamma_r)$ is nonabelian almost-simple and $P$ is a partition of $J:=G/H_{r-1}$. If $P$ is the trivial partition $P_J=\{\{j\}, j\in J\}$, the claim follows, so we  assume on the contrary it is properly coarser then  $P_J$. 
Since every intermediate subgroup between $H_{r-1}$ and $G$ is of  the form $H_i$ for some $0\leq i\leq r-1$, the correspondence between partitions and subgroups from Lemma \ref{lem:correspondence} implies that a stabilizer $H\geq H_{r-1}$ of a block in $P$ must contain $H\geq H_{r-2}$. As in the end of the proof of Theorem \ref{thm:main}, this implies that the projection of $\soc(K)$ to $\Gamma:=\Mon(f_{r-1}\circ f_r)$ is a nontrivial normal diagonal subgroup of $\Gamma_r^{d_{r-1}}$  contained in $K_{r-1}\cap \soc(\Gamma_{r})^{d_{r-1}}$.  However by Proposition \ref{prop:carmel},  $\soc(K_{r-1})=K_{r-1}\cap \soc(\Gamma_{r})^{d_{r-1}}$ is a minimal normal subgroup of $\Gamma$. Thus, $\soc(K)$ must map onto $\soc(K_{r-1})$ and $\soc(K_{r-1})\leq \Gamma_r^{d_{r-1}}$ must be a diagonal subgroup, i.e., have injective component projections. Since, by the definition of socle, $K_{r-1}$ has no nontrivial normal subgroup disjoint from $\soc(K_{r-1})$, it has injective component projections as well, i.e., is diagonal, contradicting our assumptions.
\end{proof}

\section{Stable primes and prime divisors of dynamical sequences}\label{sec:fp-cycles}
We now proceed to the proof of Corollary \ref{cor:epsilon} and Theorem \ref{cor:cycles}. We break the bulk of the proof into two propositions which apply in  a  more general setup.

In the following, let $F$ be a number field, let $f_i\in F[x]$ be indecomposable polynomials 
%
%
of degree $>1$ ($i\in \mathbb{N}$), and set $\Gamma_i:=\Mon(f_i)$,   $d_i:=\deg(f_i)$, and $d_i':=d_1\cdots d_{i-1}$, $i=1,\ldots,r$. Let $\Omega_i$ denote the splitting field of $(f_1\circ\dots\circ f_i)(x)-t$ over $F(t)$. For fixed $a\in F$, we denote by $G_m =G_{m,a}$ the Galois group of the specialization $\Omega_{m,a}/F$ of $\Omega_m/F(t)$ at $a$, as well as by $K_m=K_{m,a}$ the kernel of the projection $G_m\to G_{m-1}$, so that  $K_m\le \Gamma_m^{d_{m}'}$. Recall from \S\ref{sec: intro - methods} that $K_m$ is {\it large} if $\soc(\Gamma_m)^{d_{m}'}\le K_m$. The main idea in this section is that, for answering several key questions in arithmetic dynamics, the largeness of $K_m$ for {\it infinitely many} $m$ suffices. The relevance of this assumption has been noted before, e.g., in iterating quadratic polynomials \cite{Jones_Mandelbrot}.

\JK{In analogy to the introduction, given a sequence $(f_i)_{i=1}^\infty$ and $a\in F$, we say that $(f_i)_{i=1}^\infty$ is {\it eventually stable}  over $a$, if there exists a constant $C$ such that $(f_1\circ \dots\circ f_m)(X)-a$  has $\le C$ irreducible factors for all $m\in \mathbb{N}$. In the same vein, for a given $C>0$, and a prime $p$ of $F$, we say that $(f_i)_{i=1}^\infty$ is {\it $C$-stable} mod $p$ over $a$, if $(f_1\circ \dots\circ f_m)(X)-a$  mod $p$ (is defined and) has $\le C$ irreducible factors for all $m\in \mathbb{N}$.}

\begin{prop}
\label{lem:cycles1}
Let $\eps>0$ and  $g:\mathbb{N}\to \mathbb{N}$ a function such that $g(n)= o(n)$. Then there exists a constant $N=N(\eps)$ such that the following holds: \\
If  $f_1, f_2, f_3,\dots \in F[x]$ are of bounded degree, and if $a\in F$ is such that the kernel $K_m=K_{m,a}$ is large (in the above notation) for at least $N$ different integers $m$, then the set of primes of $F$ modulo which $(f_i)_{i=1}^\infty$ is $g(N)$-stable over $a$ is of density $<\eps$.\\
\JK{In particular, if $K_m$ is large for infinitely many $m\in \mathbb{N}$, then given any constant $C> 0$, the set of primes modulo which $(f_i)_{i=1}^\infty$ is $C$-stable over $a$ is of density $0$.}
\end{prop}

\begin{prop} \label{lem:cycles2}
Let $\eps>0$. Then there exists $N:=N(\eps)$ such that the following holds: If $f_1, f_2, f_3,\dots \in F[x]$ are of degree $\ge 5$, not linearly related over $\overline{F}$ to a monomial or a Chebyshev polynomial, if $a\in F$ is not a branch point of any $f_1\circ\dots\circ f_n$ ($n\in \mathbb{N}$) and is such that $(f_i)_{i\in \mathbb N}$ is eventually stable over $a$ and $K_m$ is large for at least $N$ different integers $m$, then the set of primes $p$ of $F$ such that $a_n:=(f_1\circ \dots \circ f_n)(a_0)\equiv a$ mod $p$ for some $n\in \mathbb{N}$ is of density $<\eps$ for every $a_0\in F$ such that $a\notin \{a_n : n\in \mathbb{N}\}$. 
 
In particular, if $(f_i)_{i\in \mathbb N}$ is eventually stable over $a$ and $K_m$ is large for infinitely many $m\in \mathbb{N}$, then this set of primes is of density $0$.
\end{prop}

\begin{proof}[Proof of Proposition \ref{lem:cycles1}]
Given any polynomial $f(x)\in F[x]$ and a prime  $p$ unramified in the splitting field of $f$, not dividing the leading coefficient of $f$ and such that the mod-$p$ reduction of $f$ is defined, 
the number of disjoint cycles of the Frobenius at $p$ is at most the number of mod-$p$ factors of $f$.\footnote{\JK{If $f$ mod $p$ is separable, one even has equality. For the general case, it suffices to note that the number of irreducible factors mod $p$ is at least the number of factors over the completion $F_p$, \DN{and the latter  is well known to coincide with the number of orbits of Frobenius.}}} In our situation, $(f_1\circ\dots\circ f_m)(X)-a$ having at most $g(N)$ irreducible factors modulo $p$, thus implies that the Frobenius at $p$ has at most $g(N)$ disjoint cycles.
 
Now fix some $k,m\in \mathbb{N}$ and let $\overline{x}\in G_{m-1}$ be an element with exactly $k$ disjoint cycles. Trivially, any preimage $x$ of $\overline{x}$ in $G_{m}$ under the projection $G_{m}\to G_{m-1}$ has at least $k$ cycles. We now assume that $K_m$ is large and estimate the proportion of such preimages having exactly $k$ cycles. 

Assume that $x$ is any such element. Denote by $\delta_{i,j}\in \{1,\dots, d_1\cdots d_m\}$ the $j$-th element of the $i$-th block in the blocks action induced by $G_{m-1}$, $i=1,\dots, d_1\cdots d_{m-1}$, $j=1,\dots d_m$. By our assumption, every cycle of $x$ has to contain every single element of the blocks involved in it. Up to relabelling elements, we may thus assume that $x$ contains a cycle of the form $(\delta_{1,1}, \dots, \delta_{r,1}, \delta_{1,2}, \dots \delta_{r,2},\dots, \delta_{r, d_m})$ for some $r\in \{1,\dots, d_1\cdots d_{m-1}\}$. 

Now consider the elements $y\circ x$, where $y$ moves through $U:=\textrm{soc}(\Gamma_m)\times \{1\}\times \dots\times \{1\} \le K_m$. (In particular, such $y$ acts nontrivially only on the first block). As mentioned above, all these elements $y\circ x$ have $\ge k$ disjoint cycles; furthermore, whenever $y(\delta_{1,2})=\delta_{1,1}$, the length of the cycle of $y\circ x$ containing $\delta_{1,1}$ is  $r$, which is strictly less than $rd_m$, i.e., $y\circ x$ has strictly more cycles than its projection to $G_{m-1}$. The proportion of such elements $y$ inside $U$ is $1/d_m \ge 1/\max_{m\in \mathbb{N}}d_m>0$. Therefore, via decomposing any coset $K_m x$ into cosets of $U$, we see that the proportion of elements of $G_m$ having the same number of cycles as their projection to $G_{m-1}$ is at most $\gamma:=1-1/\max_{m\in \mathbb{N}}d_m < 1$. Assume now that $M,N\in \mathbb{N}$ are such that $K_m$ is large for exactly $N$ different integers $m\le M$. Iterating the above for all those $m$, we see that: 
$$\frac{\#\{x\in G_M\,:\,x\text{ has }\leq g(N)\text{ cycles}\}}{\#G_M}\leq \sum_{k=0}^{g(N)-1} \gamma^{N-k}\cdot \begin{pmatrix}N\\N-k\end{pmatrix},$$ 
since indeed any single summand gives an upper bound for the probability to have exactly $k$ cycles, the binomial coefficient coming from the possible choices of indices $m$ for which the number of cycles does not grow from $G_{m-1}$ to $G_m$. Using furthermore the elementary estimate $\begin{pmatrix}N\\ k\end{pmatrix}\le \frac{N^k}{k!} \le \frac{N^{k'}}{(k')!}$ for any $k\le k'\le N$, we get 
$$\sum_{k=0}^{g(N)-1} \gamma^{N-k}\cdot \begin{pmatrix}N\\N-k\end{pmatrix} \le g(N)\cdot (\gamma^{\frac{N-g(N)}{g(N)}})^{g(N)}\cdot \frac{N^{g(N)}}{g(N)!} \le (\gamma^{N/g(N) - 1}\cdot N)^{g(N)},$$
which, due to $g(N)= o(N)$, clearly converges to $0$ as $N\to \infty$. 

Chebotarev's density theorem now implies that, for $N:=N(\eps)$ large enough, the density of primes $p$ of $F$ which are unramified in the splitting field of $(f_1\circ \dots \circ f_M)(x)-a$ and whose Frobenius has at most $g(N)$ cycles is less than $\eps$. Finally, the set of primes which do ramify in this splitting field, or divide the leading coefficient, is a finite set \DN{and hence of density $0$}, from which the assertion follows.
\end{proof}

The proof of Proposition \ref{lem:cycles2} relies on the following lemma:
\begin{lemma}\label{prop:proportion}
Suppose $\alpha, C>0$ are constants and $\Gamma_i$, $i\in\mathbb N$ are nonabelian almost simple groups for which the proportion of fixed point free elements in every coset of $\soc(\Gamma_i)$ in $\Gamma_i$, $i\in \mathbb N$, is at least $\alpha$.
Let $G_m \leq \Gamma_m\wr \cdots \wr \Gamma_1$, $m\in \mathbb N$ be a sequence of subgroups with natural projections $G_m\ra G_{m-1}$ such that each $G_m$ has at most $C$ orbits, and the kernels $K_m$, $m\in\mathbb N$ are large for infinitely many $m\in \mathbb{N}$. 
Then the proportion of fixed point free elements of $G_m$ tends to $1$ as $m\to\infty$. 
\end{lemma}

\begin{proof}
Denote by $X_m(k)$ (resp., $X_m(\ge k)$) the set of elements in $G_m$ fixing exactly $k$ (resp., at least $k$) points, and by $p_m(k)$ (resp., $p_m(\ge k)$) the proportion of the respective sets inside $G_m$. Our goal is to show that $p_m(0)\to 1$ as $m\to \infty$. 

We claim that the following estimate holds:
\begin{equation}\label{eq:cycles0}
p_m(0)\ge \sum_{j\ge 0}p_{m-1}(j)\cdot \alpha^j
\end{equation}
for all $m\ge 2$ such that $K_m$ is large. Indeed, for any such $m$, choose $x\in G_m$ such that the image of $x$ in $G_{m-1}$ has exactly $k$ fixed points, and count the fixed point free elements in the $\soc(K_m)$-coset of $x$. For $y = (y_1,\dots y_{d_1\cdots d_{m-1}})\in \soc(K_m)$, the element $x\circ y$ is fixed point free as long as, for every component $j\in \{1,\dots, d_1\cdots d_{m-1}\}$ fixed by $x$, the image of $x\circ y$ on this component is fixed point free; clearly, this occurs with probability $\ge \alpha^k$, since the component entries $y_j$ may be chosen independently (since $K_m$ is large, that is, $\soc(K_m) = \soc(\Gamma_m)^{d_1\cdots d_{m-1}}$), yielding the claim.

Assume \DN{now that for some $K>0$ and $\eps>0$ the proportion of elements with at most $K$ fixed points among all elements with a fixed point is at least $\eps$ in every $G_m$; that is,}
\begin{equation}
\label{eq:cycles1}
\exists_{K>0, \varepsilon>0} \forall_{m\in\mathbb{N}}: \dfrac{\#X_m(1) + \dots +\#X_m(K)}{\#X_m(\ge 1)} \ge \varepsilon.
\end{equation}
Then, for all $m\ge 2$ which fulfill \eqref{eq:cycles0}, we have:
$$(1-p_{m-1}(0))-(1-p_{m}(0)) = p_{m}(0)-p_{m-1}(0)\ge p_{m-1}(1)\alpha +\cdots+ p_{m-1}(K)\alpha^K\ge $$ $$\ge (p_{m-1}(1)+\dots + p_{m-1}(K))\alpha^K \ge \varepsilon \alpha^K \cdot \frac{\#X_{m-1}(\ge 1)}{\#G_{m-1}} = \varepsilon \alpha^K (1-p_{m-1}(0)),$$
or equivalently $\frac{1-p_{m}(0)}{1-p_{m-1}(0)}< 1-\varepsilon \alpha^K$, 
where the constant $\varepsilon\alpha^K>0$ is independent of $m$. Since the sequence $(1-p_{m}(0))_{m\in \mathbb{N}}$ is certainly 
monotonously decreasing, and there are infinitely many $m$ fulfilling the above by assumption, the sequence $(1-p_{m}(0))_{m\in \mathbb{N}}$ converges to $0$, so that $p_m(0)$ converges to $1$ as $m\ra\infty$. 


Assume therefore that \eqref{eq:cycles1} does not hold, i.e.
\begin{equation}
\label{eq:cycles2}
\forall_{K>0, \varepsilon>0} \exists_{m\in\mathbb{N}}: \dfrac{\#X_m(1) + \dots 
 + \#X_m(K)}{\#X_m(\ge 1)} < \varepsilon.
\end{equation}
At the same time, we may assume
\begin{equation}
\label{eq:cycles3}
\exists_{c>0} \forall_{m\in \mathbb{N}}: p_m(\ge 1) > c,
\end{equation}
or otherwise the monotonously increasing sequence $p_m(0)$ certainly converges to $1$. Then, setting $K:=\lceil2/c\rceil \cdot C$ and $\varepsilon:=c/2$ in \eqref{eq:cycles2}, we get the existence of $m\in \mathbb{N}$ such that
$$\dfrac{\#X_m(0) + (\#X_m(1) +\dots + \#X_m(\lceil 2/c\rceil C))}{\#G_m} < (1-c) + \frac{c}{2}p_m(\ge 1)\le 1-\frac{c}{2},$$
and consequently the proportion of elements having more than $2C/c$ fixed points is larger than $c/2$, so that the average number of fixed points in $G_m$ is larger than $C$. However since $G_m$ has at most $C$ orbits,
by the Cauchy-Frobenius formula that average needs to be $\le C$, yielding a contradiction.

In total, we have obtained $p_m(0)\to 1$ in all cases.
\end{proof}

\begin{proof}[Proof of Proposition \ref{lem:cycles2}]
%
Set $a_n=(f_1\circ\dots\circ f_n)(a_0)$ for $n\in \mathbb N$. Firstly, for any given $m\in \mathbb{N}$, 
there can be only finitely many primes $p$ such that $a_n\equiv a$ mod $p$ for some $n$ {\it less than} $m$,  due to the assumption $a\notin \{a_n: n\in \mathbb{N}\}$.
On the other hand, all {\it other} primes $p$ for which $a_n\equiv a$ mod $p$ for some $n(\ge m)$ are such that $(f_1\circ \dots \circ f_m)(x)-a$ has a root modulo $p$.

Under the assumption that $a$ is a non-branch point of $f_1\circ\dots\circ f_m$, the splitting field of $(f_1\circ \dots \circ f_m)(x)-a$ equals the specialization $\Omega_{m,a}/F$, see, e.g., \cite[Lemma 2]{KN18}. Moreover, exempting finitely many primes $p$ (namely, those ramifying in $\Omega_{m,a}/F$), the condition $a_n\equiv a$ mod $p$ for some $n\ge m$ implies that the Frobenius at $p$ in $G_m$ has at least one fixed point, and hence that $(f_1\circ\dots\circ f_m)(x)-a$ has a root modulo $p$. 
 Given {\it any} $m\in\mathbb N$, it thus follows that the density of primes $p$ for which $a_n\equiv a$ mod $p$ for some $n$ is bounded from above by the proportion of elements with a fixed point in $G_m$. 
 It therefore suffices to prove that, under our assumptions, the proportion of fixed point free elements in $G_m$ tends to $1$ as $m\to \infty$.

 This follows from Lemma \ref{prop:proportion}, once we show that there exists a constant $\alpha>0$ such that the number of fixed point free elements in every coset of $\soc(\Gamma_i)$ in $\Gamma_i$, $i\in \mathbb N$, is at least $\alpha$. We claim that we can choose $\alpha=\frac{1}{4}$, thus concluding the proof. Indeed, by the classification achieved in \cite{Mul}, all primitive monodromy groups of nonsolvable polynomials of degree $>31$ are alternating or symmetric (yielding a value for $\alpha$ rapidly approaching $1/e$, since the proportion of fixed point free elements in the two cosets of $A_n$ is well known to be  $\sum_{i=0}^n \frac{(-1)^i}{i!} \pm \frac{n-1}{n!}$, cf.\ \cite{Olds}), 
 whereas for primitive monodromy groups of degree $\le 31$,\footnote{Here, we use $d_i\ge 5$ and the assumption that $f_i$ is not linearly related to a monomial or a Chebyshev polynomial, since indeed the claim becomes false for $\Gamma_i=S_4$ or $\soc(\Gamma_i)=C_p$.} a direct computer check gives the lower bound $\alpha=1/4$, with the minimal value attained exactly for the group $\Gamma_i = \PGammaL_2(9)\le S_{10}$. 
\end{proof}

\begin{remark}
\begin{itemize}
\item[a)] The above proof can be adapted to composition of not necessarily polynomial maps $f_i: X_i\to X_{i-1}$ of curves $X_i$ with almost simple monodromy groups, up to some technical conditions. Notably, the condition on the proportion of fixed point free elements in cosets of $\soc(\Gamma_i)$ being bounded away from $0$, which is automatically fulfilled in the case of polynomial maps, should be added as an extra assumption. 
\item[b)] 
Furthermore, the proof readily  yields a positive characteristic analog as long as the polynomials involved define tamely ramified separable function field extensions. This is because of the group-theoretic nature of the arguments; in particular, all underlying classification results for monodromy groups of polynomials, as well as Ritt's theorem carry over from $\mathbb{C}$ to $\overline{\mathbb F}_p$ as long as the condition of tame ramification is imposed, cf.\ \cite{DW72}. 
\end{itemize}
\end{remark}

\begin{proof}[Proof of Theorem \ref{cor:cycles}]
Assertion \DN{1)} of the theorem is immediate from the last assertion of Proposition \ref{lem:cycles1}, 
upon setting $f_n=f$ for all $n\in \mathbb{N}$. Similarly, Assertion \DN{2}) is immediate from the last assertion of Proposition \ref{lem:cycles2}.
\end{proof}

\begin{proof}[Proof of Corollary \ref{cor:epsilon}]
This follows from Proposition  \ref{lem:cycles1} and \ref{lem:cycles2} as soon as we can make sure that, for every given $m\in \mathbb{N}$, the set of values $a$ which do not yield a large $K_m$ is a thin set. This, however, is asserted by Corollary \ref{cor:spec}.
\end{proof}

\JK{\begin{remark}
\label{rem:finite_index}
Note that in the event of large kernel for all $m$, the extra assumptions $a\notin \{f^{\circ n}(a_0): n \in \mathbb{N}\}$ and ``$f$ eventually stable over $a$" in Theorem \ref{cor:cycles} are automatically fulfilled; indeed, since the socle of a primitive group is necessarily transitive, this assumption forces $\Image\rho_{f,a}^{(n)}$ to be transitive, and hence $f^{\circ n}(x)-a$ to be irreducible for all $n\in \mathbb{N}$.
\end{remark}}

\section{Largeness of arboreal representations}
\label{sec:spec}
\JK{The previous section has left open whether, for given sequences $(f_n)_{n\in \mathbb{N}}$ of polynomials, there exist values $a$ for which the kernels $K_m=K_{m,a}$ are large for infinitely many $m$, as assumed in Theorem \ref{cor:cycles}.
In Proposition \ref{prop:spec}, given an (infinite) sequence of \DN{polynomial} compositions \DN{with bounded number of branch points}, we will provide general sufficient conditions for such largeness of kernels.}
In particular this provides examples for situations in which the assumptions of the previous section are fulfilled. 

Recall \cite{Dix} that a normal subgroup $N\lhd G$ is {\it invariably generated} (in $G$) by subgroups $I_1,\ldots,I_r\leq G$ if  the subgroup generated by the conjugates $I_1^{x_1},\ldots,I_r^{x_r}$ contains $N$ for 
any $x_1,\ldots,x_r\in G$. 
\begin{prop}\label{prop:spec}
Let $F$ be a number field and $N\in\mathbb N$. Let $f_1,f_2,\ldots$ be a sequence of indecomposable polynomials in $F[x]$  
with $\deg f_i\leq N$, and set $\Gamma_r:=\Mon(f_r)$. 
Assume that all of the following hold:
\begin{itemize}
\item[i)] The union of branch points of $g_r:=f_1\circ\cdots\circ f_r$, $r\in \mathbb N$ is finite.
\item[ii)] For each $r\in \mathbb{N}$, $f_r\circ f_{r+1}$ has at least one inertia subgroup whose intersection with $\ker(\Mon(f_r\circ f_{r+1})\to \Gamma_r)$ is nondiagonal in $\Gamma_{r+1}^{\deg(f_r)}$.
\item[iii)]
The socle of $\Gamma_r$ is nonabelian and is invariably generated in $Aut(\Gamma_r)$ by the inertia subgroups of $f_r$, for all $r\in \mathbb{N}$. 
\end{itemize}
Then there exist infinitely many $\alpha\in F$ such that the Galois group of the fiber of $g_r$ over $t\mapsto\alpha$  contains all composition factors 
of $[\soc(\Gamma_i)]_{i=1}^r$ for all $r\in\mathbb N$.
\end{prop}

\begin{remark}\label{rem:cond}
\DN{
Assumption ii) is automatically fulfilled in important key cases, see Lemma \ref{lem:nondiag}, and also holds for generic choices of polynomials. Assumption iii) is generically fulfilled from a group-theoretic viewpoint:  namely, it holds for most inertia configurations of polynomials since a a $d$-cycle and a random permutation (coming from inertia groups at infinity and another point) generate $A_d$ with probability tending to $1$, as $d\to \infty$ \cite{LP}. Many such examples are given in Example \ref{exam:shabat}(1).  Assumption i) on the other hand does not hold generically, and hence is the most restrictive among the assumptions.  The assertion is expected to hold though far beyond this assumption, see Remark \ref{rem:optimist}.2.}
\end{remark}
We shall use the relation between the two diagonality notions from \S\ref{sec:ker}:
\begin{lemma}\label{lem:diagvsdiag} Let $V$ act transitively on $J$, let $G\le U\wr_J V$ project onto $V$, set $K:=G\cap U^J$ and let $U_0\trianglelefteq U$ the image of $K$ under projection to a component. If $Z(\soc(U_0))=1$ 
and $K$ is a diagonal subgroup, then every $x\in K$ is diagonal in $U_0^J$.\end{lemma}
\begin{proof} Let $\pi_j:K\to U_0$, $j\in J$ denote the component projections. Since $K$ is diagonal, $\pi_{j_1}\circ \pi_{j_2}^{-1}\in \Aut(U_0)$  sends $x_{j_2}$ to $x_{j_1}$ for any $x=(x_1,\dots x_J)\in K$. But since $Z(\soc(U_0))=1$, 
$Aut(U_0)$ embeds into $Aut(\soc(U_0))$, showing that $x$ is diagonal. 
\end{proof}

\begin{proof}[\DN{Proof of Proposition \ref{prop:spec}}]
Let $A_r:=\Mon_F(g_r)$ for $r\in\mathbb N$. 
By Corollary \ref{cor:pols} applied over $\oline F$, we see that 
$\Mon_{\oline F}(g_r)$ contains all composition factors of $[\soc(\Gamma_i)]_{i=1}^r$. 
Note also that, since $f_r$ is not linearly related to $x^d$ or $T_d$, $d\in\mathbb N$,  it follows as in the proof of Corollary \ref{cor:pols} that $\Gamma_r$ is a nonabelian almost simple group.

We will next apply Theorem \ref{thm:ram_in} in combination with the field version of Corollary \ref{lem:consecutive} to ensure that at infinitely many places $t\mapsto\alpha$, the  Galois group $A_{r,\alpha}=\Gal(g_r(x)-\alpha, F)\leq A_r$ 
also contains the composition factors of $[\soc(\Gamma_i)]_{i=1}^r$ for all $r\in \mathbb{N}$.  
To this end, let $T=\{t_1,\ldots,t_m\}$ be the finite union of all branch points of $g_i$, $i\in \mathbb N$. 
By Chebotarev's density theorem, we may pick distinct primes $p_1,\ldots,p_m$ of $F$ 
such that $F(t_j)$ has a degree $1$ prime over $p_j$ for every $t_j\in T$, and the primes are outside the finite set of primes consisting of: A) the primes dividing $N!$ (and hence the primes dividing $\#A_r$ for any $r$), B)  the set of primes in case ii) of Theorem \ref{thm:ram_in} (a finite set depending only on $T$), and C) the exceptional set $\mathcal{S}_0$ of primes from Theorem \ref{thm:ram_in} for any $g_r(x)-t$, $r\in \mathbb{N}$. Note that since $A_r$ has a trivial center, the set in C) is explicitly described in Proposition \ref{prop:spec}. 

As   $F(t_j)$ has a degree $1$ prime over $p_j$,  we may pick, using the Chinese Remainder Theorem, (infinitely many) $\alpha\in F$ such that $\alpha$ meets $t_j$ at $p_j$ with multiplicity $1$ for every $j=1,\ldots,m$. 

Fix $r\in \mathbb{N}$ and let $F_i$ and $F_{i,\alpha}$ denote  the fixed field of a point stabilizer in the action of $A_i$ and $A_{i,\alpha}$, respectively, so that $F_i$ and $F_{i,\alpha}$ are the root fields of $g_i(x)-t$ and $g_i(x)-\alpha$, resp., $i=1,\ldots,r$. By possibly replacing the choice of points, assume $F_i\subseteq F_{i+1}$ and $F_{i,\alpha}\subseteq F_{i+1,\alpha}$ for $i=1,\ldots,r$. Let $\Omega_i$ denote the Galois closure of $F_i/F(t)$. 
For $i\in \{1,\dots, r\}$ and $j\in \{1,\dots,m\}$, denote by $(I_{t_j})_i$ the image of the inertia group of $g_r(x)-t$ at $t_j$ under projection $A_r\to A_i$. Since $\alpha$ meets $\infty$ at some $p_j$,  Theorem \ref{thm:ram_in} implies that  $A_{i,\alpha}$ contains a conjugate of $(I_\infty)_i$, and hence is transitive since $(I_\infty)_i\leq A_{i}$ is transitive. In particular, the extensions $F_{i,\alpha}/F_{i-1,\alpha}$, $i=1,\ldots,r$ are independent of the choice of point and block stabilizers up to $F$-isomorphism.

Let $\Omega'_{i,\alpha}$ denote the Galois closure of $F_{i,\alpha}/F_{i-1,\alpha}$, $i=1,\dots, r$, and set $\Gamma_i'=\Gal(\Omega'_{i,\alpha}/F_{i-1,\alpha})$, which may be identified with a subgroup of $\Gamma_i$.
Let $t_j\in T$ and let $t_j'\in g_{i-1}^{-1}(t_j)$ be any preimage. Since $\alpha$ meets $t_j$ at $p_j$ with multiplicity $1$, Theorem \ref{thm:ram_in} implies that $A_{i,\alpha}$ contains an $A_i$-conjugate  of the inertia group $(I_{t_j})_i$.
By transitivity of $A_{i,\alpha}$, it thus also contains, for any $A_i$-conjugate $U$ of $\textrm{Gal}(\Omega_i / F_{i-1})$, an $A_i$-conjugate of $(I_{t_j})_i\cap U$. However these intersections are (again, up to conjugation in $A_i$) exactly the inertia subgroups at preimages of $t_j$ in $\Omega_i/ F_{i-1}$. 
\JK{Note that the groups $U$ are exactly the stabilizers of a block arising from the decomposition $g_i = g_{i-1}\circ f_i$. Now use a suitable element of the cyclic transitive subgroup of $A_{i,\alpha}$ inherited from $(I_\infty)_i$, in order to conjugate back into the fixed block stabilizer $\textrm{Gal}(\Omega_i / F_{i-1})$.}
  Then, project to the Galois groups of (the Galois closures of) $F_{i,\alpha}/F_{i-1,\alpha}$ and of $F_{i}/F_{i-1}$, respectively, to see that $\Gamma_i'$ contains, up to an automorphism of the almost simple group $\Gamma_i$, a \JK{copy} of each inertia group in (the Galois closure of) $F_{i}/F_{i-1}$, i.e., of each inertia group of $f_i$. 
Since $\soc(\Gamma_i)$ is invariably generated \DN{in its automorphism group} by these inertia groups, we get that $\Gamma_i'\supseteq \soc(\Gamma_i)$  is almost simple for all $i=1,\dots, r$.

Since $A_{i,\alpha}$ contains a conjugate of $(I_\infty)_i$, the kernel $K$ of the projection $A_{i,\alpha}\ra A_{i-1,\alpha}$ is nontrivial. Finally, let $\Delta_i \le \Gamma'_{i+1} \wr \Gamma'_{i}$ denote the Galois group of the Galois closure of $F_{i+1,\alpha}/F_{i-1,\alpha}$ for $i=1,\dots, r-1$ and $K_i\triangleleft \Delta_i$ the kernel of the projection $\Delta_i\to \Gamma'_i$. Let $I$ be an inertia group of $f_{i}\circ f_{i+1}$ whose intersection with $\ker(\Mon(f_{i}\circ f_{i+1})\to \Gamma_i)$ is nondiagonal in \DN{$\Gamma_{i+1}^{\deg f_i}$}, as provided by Assumption ii). By the same argument as above, $\Delta_i$ contains a \JK{$\Mon(f_i\circ f_{i+1})$-}conjugate of $I$, and since \JK{conjugation in $\Mon(f_i\circ f_{i+1})$ induces an automorphism of $\Gamma_{i+1}^{\deg(f_i)}$, it leaves} the diagonality property \JK{in this direct product} invariant, as  in the end of \S\ref{sec:ker}. \JK{Thus,} $K_i$ contains a nondiagonal element, \JK{and consequently} the kernel $K_i\triangleleft \Delta_i$ is nondiagonal by Lemma \ref{lem:diagvsdiag}.

Hence, \DN{the assumptions of the field version of Corollary \ref{lem:consecutive}, stated in Appendix \ref{app:fields}},  
 are satisfied for the chain $F\subset F_{1,\alpha}\subset\dots\subset F_{r,\alpha}$, and it thus follows that $A_{r,\alpha}$ contains all the composition factors of $[\soc(\Gamma_i)]_{i=1}^r$, concluding the proof.
\end{proof}

The conditions of the proposition hold for iterates in the following scenario. Recall that $\Br(f)$ and $\Ram(f)$ denote the sets of branch points and ramification points of $f$, resp.
\begin{cor}\label{cor:dynamics}
Let $f\in F[x]$ be an indecomposable polynomial of degree $d$ with monodromy group $\Gamma:=\Mon(f)$.
Assume that $\Br(f)\subseteq \Ram(f)$, that $f$ is not linearly related to $x^d,T_d$ over $\oline F$,  and that $\soc(\Gamma)$ is invariably generated inside $\Aut(\Gamma)$ by the inertia groups of $f$. Then there exist infinitely many $\alpha\in F$ such that the Galois group of  $f^{\circ r}(x)-\alpha$ contains the composition factors of $[\soc(\Gamma)]^r$ for all $r\in\mathbb N$. 
\end{cor}
The corollary follows immediately from the proposition and the following lemma. 
\begin{lemma}
\label{lem:nondiag}
    Suppose $f_1,f_2\in F[x]$ are two indecomposable degree-$d$ polynomials of the same ramification type, that are not linearly related to $x^d$ over $\oline F$,  and  that satisfy  $\Br(f_2)\subseteq \Ram(f_1)$.  Then there is at least one inertia subgroup of $\Mon(f_1\circ f_2)$  whose intersection with the kernel of the projection $\Mon(f_1\circ f_2)\ra \Mon(f_1)$ is nondiagonal in $\Mon(f_2)^{\deg(f_1)}$.
\end{lemma}
\begin{proof}
Let $\tilde f:\tilde X\ra\mP^1$ denote the Galois closure of $f:=f_1\circ f_2$. 
Note first that the indecomposability of $f_1$ forces the g.c.d.\ $e=\gcd(E_{f_1}(P))$ of ramification indices in the ramification type  $E_{f_1}(P)$ to be $1$.
Indeed, otherwise $e>1$ and by composing with a linear polynomial in $\oline F[x]$ which sends $P$ to $0$ we obtain a polynomial that factors through $x^e$, contradicting the assumption that $f_1$ is indecomposable and is not linearly related over $\oline F$ to $x^e$.  In particular it follows that $d>2$. 

Denote by $\sigma_2$ an element of maximal order $N$ among generators of inertia groups over finite points in $\Br(f_2)$, and let $Q$ be the corresponding branch point. Let $P=f_1(Q)$. 
Let $\sigma$ (resp.\ $\sigma_1$) be a generator of an inertia group of $f$ (resp.\ $f_1$) over $P$, \JK{and let $m$ be \DN{the order of} $\sigma_1$}. 
Since $Q\in \Br(f_2)\subseteq \Ram(f_1)$, \JK{the ramification index of $Q$ over $P$ is larger than $1$, and in particular the order of $\sigma$}  is necessarily a proper multiple of $N$. Since the ramification of $f_1$ and $f_2$ coincide, we have $m\leq N$, and hence  $\sigma^m$ is a nontrivial element in the block kernel.  

Assume on the contrary all inertia subgroups of $f_1\circ f_2$ are such that their intersection with $K:=\ker(\Mon(f_1\circ f_2)\ra \Mon(f_1))$ is diagonal in $\Mon(f_2)^d$. Then $\sigma^m\in K$ is diagonal, so that 
its restriction to each block is nontrivial. In particular, every preimage in $f_1^{-1}(P)$ is a branch point of $f_2$. 
Since $\Br(f_2)\subseteq \Ram(f_1)$, it follows that all points in $f_1^{-1}(P)$ are ramified under $f_1$. 

To reach a contradiction,  consider the Riemann--Hurwitz contribution  under $f_2$.  Since the ramification of $f_2$ is the same as that of $f_1$, there is a point $P'$ whose ramification under $f_2$ is the same as that of $f_1$ over $P$, and hence its  Riemann--Hurwitz contribution $R_{f_2}(P')$ coincides with  $R_{f_1}(P)=d-\#f_1^{-1}(P)$. 
We first claim that $P'\in f_1^{-1}(P)$. 
Since, by the above, every point in $f_1^{-1}(P)$ is a branch point for $f_2$,  the RH-contribution $\sum_{Q'\in f_1^{-1}(P)} R_{f_2}(Q')$ is at least $\#f_1^{-1}(P)$. \DN{Thus, if $P'\notin f_1^{-1}(P)$, then} the total RH-contribution for $f_2$ is at least 
$$ R_{f_2}(\infty) + R_{f_2}(P') + \sum_{Q'\in f_1^{-1}(P)} R_{f_2}(Q')\geq (d-1)+(d-\#f_1^{-1}(P)) + \#f_1^{-1}(P) = 2d-1,$$
exceeding the total RH-allocation of $2d-2$, yielding a contradiction. 

Henceforth assume $P'\in f_1^{-1}(P)$. The \DN{following similar argument shows that $Q$ can be picked to be $P'$}, and hence $\sigma_1$ and $\sigma_2$ are conjugate in $S_d$. Here,
\begin{align*}\label{equ:RH}
R_{f_2}(\infty)+R_{f_2}(P')+\sum_{Q'\in f_1^{-1}(P)\setminus\{P'\}}R_{f_2}(Q') \geq & (d-1)+(d-\#f_1^{-1}(P)) \\ & + (\#f_1^{-1}(P)-1)  = 2d-2,    
\end{align*}
which is the total RH-allotment for $f_2$. Thus, the last inequality is an equality, $f_2$ is unramified away from $f_1^{-1}(P)\cup\{\infty\}$, and $R_{f_2}(Q')=1$ for every $Q'\in f_1^{-1}(P)\setminus\{P'\}$. Thus  every  $Q' \in f_1^{-1}(P)\setminus\{P'\}$ is simply branched under $f_2$, \DN{that is, has a single ramified preimage and that preimage has ramification index $2$}. 
\JK{By the maximality assumption in the definition of $Q$, this implies that the inertia groups of $P'$ and $Q$ under $f_2$ have the same order, i.e., $m=N$ and we may after all assume} $P'=Q$.

 

\JK{Now consider again $\sigma^m\in K$. This is a nontrivial diagonal element, whence its projection to any block is of the same order $h\ge 2$. But the projections to the blocks are all conjugate to powers of the inertia group generators under $f_2$ of points in $f_1^{-1}(P)$. Since some of these are transpositions, we have $h=2$.}
This forces \JK{both $\sigma_1$ and $\sigma_2$} to be of even order, and moreover the ramification index  $e_{f_1}(Q'/P)$ under $f_1$ to be even for every $Q'\in f_1^{-1}(P)$ 
(since otherwise, taking the $m$-th power would lead to trivial image on each of those blocks). 

We have thus shown that $2\,|\,e_{f_1}(Q'/P)$ for every $Q'\in f_1^{-1}(P)$, contradicting our first observation that  $\gcd(e_{f_1}(Q'/P)) =1$, where $Q'$ runs over $f_1^{-1}(P)$.
\end{proof}
The following construction leads to infinite families of polynomials $f$ fulfilling the assumptions of Corollary \ref{cor:dynamics}.
\begin{example}\label{exam:shabat}
(1) 
Pick a random element $\sigma\in S_d$ and let $R_\sigma$ denote its cycle structure.
We claim that the probability that there exists a genus $0$ product one triple consisting of a $d$-cycle, $\sigma$ and a third element $\tau\in S_d$, which invariably generate a subgroup containing $A_d$, tends to $1$ as $d$ goes to infinity. For each such tuple there exists a polynomial map $f:\mP^1_\C\ra\mP^1_\C$ with ramification type $[d],R_\sigma,R_\tau$ by Riemann's existence theorem. 
By composing $f$ with linear polynomials we may move the two finite branch points arbitrarily and thus assume $\Br(f)\subseteq \Ram(f)$. 
Each such polynomial, a.k.a.\ a ``Shabat polynomial" \cite[\S2.2]{LZ}, is defined over some number field and satisfies the conditions of Corollary~\ref{cor:spec}.  

To prove the claim, 
we construct $\tau\in S_d$ such that  $\sigma\tau$ is an $d$-cycle and $\tau$ is a product of $r-1$ transpositions, where $r=\#R_\sigma$ is the number of nontrivial disjoint cycles in $\sigma$. \DN{As in \S\ref{sec:def}, $f$ then has three branch points $P_0,P_1,P_2$} with $E_f(P_0)=[d],E_f(P_1)=R_\sigma$ and $E_f(P_2)=R_\tau$. The Riemann--Hurwitz formula then implies that the genus \DN{$g$ of the source of $f$ is $0$, for:}
\begin{align*}
    2(g-1) & =-2d+ R_f(P_0) + R_f(P_1) + R_f(P_2) \\ 
    & = -2d + (d-1) + (d-r) + (r-1)=-2. 
\end{align*}
To construct $\tau$, write $\sigma$ as a product $\sigma_1\cdots\sigma_r$ of nontrivial disjoint cycles, 
$ \sigma_i=(x_{i,1},\dots, x_{i,k_i})$, $i=1,\ldots,r$. 
Letting $\tau$ be the product of transpositions $(x_{i,k_i}, x_{i+1,1})$, $i=1,\ldots,r-1$, we see that $\sigma\tau$ is an $r$-cycle, as needed.
Finally, it is known  \cite{LP} that, as $d\to \infty$, the probability for a random element $\sigma\in S_d$ to lie in any transitive subgroup other than $A_d$ or $S_d$ tends to $0$. Hence in particular, the probability for a random element $\sigma$ and an $d$-cycle to invariably generate a subgroup containing $A_d$ tends to $1$ as $d\to\infty$, as required. \\
(2) We give an example of a 
 large family of examples which is additionally defined over $\mathbb{Q}$. It is the family of degree-$d$ polynomial maps of ramification type $[d], [r,1^p], [s^q,t]$, where $d=r+p=qs+t = t+r+q(s-1)-1$. Indeed, the Riemann-Hurwitz formula yields immediately that this is a polynomial ramification type, and it can be turned into a ``dynamical Belyi map" (one with branch points $0$,$1$ and $\infty$, each of which maps to itself), upon suitably prescribing the unique critical points of ramification index $r$, $t$ and $d$ over each branch point by applying a M\"obius transformation. More importantly, \cite[Prop.\ 2.2.19]{LZ} states that the corresponding polynomial map is defined over $\mathbb{Q}$. 
 
For our purposes, it then remains to verify that elements of the three cycle types invariably generate $A_d$ or $S_d$. Note that $s$ and $t$ must necessarily be coprime: \DN{for, otherwise $f$ is decomposable as in the proof of Lemma \ref{lem:nondiag}, and the generated group is imprimitive}. To verify that this condition is also sufficient, 
note that Lemma \ref{lem:primitive} below asserts that the generated group is primitive.  Since it also contains a cycle of length $2\le k\le d-3$, \cite[Corollary 1.3]{GJones} then implies that it contains $A_d$. Thus, the claim follows from:
\end{example}
\begin{lemma}\label{lem:primitive}
    Suppose $x,y,z\in S_d$ have cycle structures $[d], [r, 1^{d-r}], [s^q,t]$, resp., for $q,s,t,r\in\mathbb N$, $r,t>1$ with $q=(d-t)/s$  and $(s,t)=1$. Then  $\langle x,y,z\rangle$ is primitive. 
\end{lemma}
\begin{proof}
Clearly $G:=\langle x,y,z\rangle$ is transitive as it contains a $d$-cycle. Suppose $P$ is a $G$-invariant partition into blocks of size $m$ dividing $d$.
By the Riemann--Hurwitz formula $r+t+q(s-1) = d+1$. Since $t+qs=d$, this yields $r=q+1$. 


First note that since $z^s$ is a $t$-cycle which preserves $P$, either $m\divides t$ or ($t< m$ and $z^s$ preserves the blocks of $P$). Similarly either $m\divides r$ or $r< m$.

If $m\divides t$, then $m\divides d-t=qs$. As $(s,m)=1$, this gives $m\divides q$. Thus the above equality $r=q+1$ gives $r=1$ mod $m$. As noted above this forces $r=1$,   contradicting $r>1$. 

Henceforth assume $t<m$. 
If $m\divides r$, the equality $r=q+1$ yields $q=-1$ mod $m$. Since $d=t+sq$, this gives \DN{$s=t$ mod $m$. We claim that $s<m$, and hence $s=t$, contradicting $(s,t)=1$ and $t>1$. To see the claim observe that $z^t$ consists of $q$ $s$-cycles and $t$ fixed points. A block $p\in P$ which contains such a fixed point, is fixed by $z^t$. Since $t<m$, $p$ contains also non-fixed points which belong to an $s$-cycle. As $z^t$ fixes $p$, that whole $s$-cycle is contained in $p$, and hence $s<m$, yielding the claim.} 

Henceforth assume $r<m$.  Since $z^s$ preserves the block system, so is the $t$-cycle in $z$. Thus, there is a positive integer $\DN{v}<m$ such that $v$ of the $s$-cycles in $z$ preserve the block which contains that $t$-cycle, so that $m=vs+t$. Thus 
$$(q-v)s\equiv (q-v)s+vs+t=qs+t=d=0\text{ mod }m.$$
As $(s,m)=1$, we get $q=v$ mod $m$. Since $v<m$ and $q=r-1<m$, this yields $q=v$. Thus $d=qs+t=vs+t=m$, so that the partition is trivial. 
\end{proof}
\DN{Finally} note that the subcase $s=1$ leads to the family of (polynomial) ``normalized Belyi maps", for which \cite{BEK} already shows that the image of the arboreal Galois representation indeed  contains $[A_d]_{i=1}^\infty$ for infinitely many specializations $t\mapsto\alpha$.

\begin{remark}\label{rem:optimist}
1) We note furthermore that the assumption of invariable generation in Corollary \ref{cor:dynamics} is fulfilled for most of the ``exceptional" indecomposable polynomials $f$, i.e., for most $f\not\sim x^n,T_n$ with $\Mon(f)\ne A_d, S_d$. Indeed, using the classification obtained in \cite{Mul}, one verifies directly, e.g. with Magma, that the only cases violating the assumption are those of degree $9$ with $\Mon(f) = P\Gamma L_2(8)$.\\
2) We expect analogues of the proposition and corollary to hold for more general maps, as long as the conditions of Corollary \ref{cor:pols iterative} hold. 
It is even conceivable that, given a reasonably general polynomial $f$, the question whether some specialization $t\mapsto a$ preserves the full monodromy group $\textrm{Mon}(f^{\circ n})$ for all $n$ simultaneously is decided on a certain finite level $n_0$, in which case Hilbert's irreducibility theorem would imply that the set of exceptional specialization values is a thin set after all. Cf.\ \cite{Hindes}, which obtains (conditionally on Vojta's conjecture) such a result for iterates of certain quadratic maps.
\end{remark}

\section{Splitting}\label{sec:splitting}  Theorem \ref{thm:main-intro} follows from
Corollary \ref{cor:pols iterative} and the following splitting assertion:
\begin{prop}\label{prop:splitting}
    Suppose $G\leq [S_{d_i}]_{i=1}^r$ contains the composition factors of $[A_{d_i}]_{i=1}^r$, for integers $d_i\geq 5$, $i=1,\ldots,r$. Then $G$ contains a subgroup isomorphic to $[A_{d_i}]_{i=1}^r$. 
\end{prop}


The proof uses the following \DN{elementary} lemma. Denote  $[d]=\{1,\ldots,d\}$. 
\begin{lemma}\label{lem:induced} 
For $d\geq 5$, suppose the natural projection $\pi:G\ra A_d$ from $G\leq C_2\wr_{[d]} A_d$ is onto. 
Then $\pi$ splits. 
\end{lemma}
For this we use Shapiro's lemma \cite[\S I.2.5]{Ser}. For abelian $U$, it gives an isomorphism $\HG^i(V,\Ind_T U)\cong \HG^i(V_1, U)$ of cohomology groups for $i\geq 1$, where $V_1\leq V$ is a point stabilizer in the action on $T$.  
\begin{proof}[Proof of Lemma \ref{lem:induced}]
The kernel $K:=\ker\pi\leq \Ind_JC_2$ is a $G$-invariant subgroup. Since $A_d$ is transitive on $J$, each of the projections $\pi_j:K\ra C_2$, $j\in J$ has the same image, cf.\ \S\ref{sec:ker}. 
The possibilities are $K=1,D,I,\Ind_JC_2$, where $D$ is the diagonal subgroup consisting of vectors with equal entries, and $I$ is the augmentation subgroup consisting of vectors whose sum is $0$, see \cite{Mor}. In case $K=1$ or $\Ind_JC_2$, our group $G$ is $A_d$ or the full group, resp., in which case $\pi$ clearly splits.

In case $K=I$, the inclusion $K\leq \Ind_JC_2$ has cokernel $C_2$ and hence induces the following part of the long exact sequence 
\begin{equation}\label{equ:long} 1=\HG^1(A_d,C_2)\ra \HG^2(A_d,I)\ra \HG^2(A_d,\Ind_JC_2).
\end{equation}
Since  $G\cdot \Ind_JC_2$ is the full group $C_2\wr_JA_d$, the image of the extension $\eps_G\in \HG^2(A_d,I)$, induced by $G$, in $\HG^2(A_d,\Ind_JC_2)$ is trivial. Since this map is injective by \eqref{equ:long}, we see that $\eps_G=1$, as required. 

In case $K=D$ is diagonal, $G$ is a central extension of $C_2=\{0,1\}$ by $A_d$. The only such \JK{nonsplit} extension is the Schur cover $\hat A_d$. 
Assume on the contrary $G\cong\hat A_d$ so that  there is an embedding $j:\hat A_d\ra C_2\wr A_d$, \JK{and in particular a faithful action of $\hat A_d$ on $2d$ points. The point stabilizer of such an action would have to be an index $d$ subgroup of $A_d$ (i.e., a group isomorphic to $A_{d-1}$), split in $\hat A_d$, which is impossible, since even the double transpositions of $A_d$ are nonsplit in $\hat A_d$.}

\end{proof}
\DN{In the case $K=I$ the triviality of the extension can be seen also from the fact that the {\it heart} $I/I\cap D$ of $\Ind_{A_1}^A C_2$ is trivial by \cite[Corollary 1]{KP}. In the case $K=D$, the nonexistence of a lift $\hat A_d\ra C_2\wr A_d$ of the projection $\hat A_d\ra A_d$ can also be deduced from the nonexistence of $\hat A_d$-representations of degree $<d(d-1)$, see \cite[Thm.\ 1.1]{GL}.}
\begin{remark}\label{rem:split}
The proof of Lemma \ref{lem:induced} applies more generally when $A_d$ is replaced by a  doubly transitive simple group $A$ and $C_2$ is replaced by a cyclic group $C_p$ of prime order $p$ as long as the point stabilizer $A_1$ is \JK{nonsplit} in any nonsplit central extension of $A$ by $C_p$, and the {\it heart} of $\Ind_{A_1}^A C_p$  is an irreducible $A$-module. 
For the list of doubly transitive groups with trivial heart see \cite{Mor}.
\end{remark}
Moreover, when $A_d$ is replaced by \JK{an arbitrary} simple group $A$, the assertions of 
 \DN{Proposition \ref{prop:splitting} and Lemma \ref{lem:induced}} may fail:
\begin{example}\label{exam:non-split}
Let $\hat A$ be a nonsplit central extension of a simple group  $A$ by $C_2$ such that the point stabilizer $A_1$ splits in $\hat A$ (as opposed to $A_{d-1} $ being nonsplit in $\hat A_d$). In this case the action of $A$ on $J$ can be extended to $\hat A$ and $\hat A$ acts on $\hat A/A_1$ imprimitively inducing an embedding of $\hat A$ into $C_2\wr A$. Thus the resulting embedding of $A^J\rtimes \hat A$ into $\hat A\wr A$ contains the composition factors of $\hat A\wr A$ but no split copy of $A$. 
\end{example}
The rest of the proof in this section works when $A_d$ is replaced by a simple group $A$ as in Remark \ref{rem:split} when in addition one assumes $A$ is almost complete, that is, the projection $\Aut(A)\ra\Out(A)$ splits. 
We first extend the lemma as follows: 
\begin{cor}\label{cor:induced}
    For $d\geq 5$ and $k\geq 1$, suppose the natural projection $\pi:G\ra A_d$ from $G\leq C_2\wr_{[k]\times [d]} A_d$ is onto, where the actions of  $A_d$ on $[d]$ and $[k]$ are the standard and trivial ones, resp. Then $\pi$ splits. 
\end{cor}
\begin{proof}
    Our group $G$ induces a group extension of $A_d$ and $K:=\ker\pi\leq C_2^{[k]\times[d]}$, yielding a class $\alpha\in \HG^2(A_d,K)$.  It suffices to show that $\alpha$ is trivial. 
    

We argue by induction on $k$. The induction base $k=1$ follows from Lemma \ref{lem:induced}. Consider the kernel  $K_1'$ and image $K_1$ of the projection $G\cap C_2^{[k]\times [d]}\ra C_2^{\{1\}\times [d]}$. The projection induces a group extension of $A_d$ by $K_1$ which embeds into $C_2\wr_{[d]} A_d$. Thus Lemma \ref{lem:induced} shows that the image of $\alpha$ in $\HG^2(A_d,K_1)$ is trivial.  The long exact sequence corresponding to $K_1'\ra K\ra K_1$ shows that $\alpha$ is the restriction $i_1^*(\alpha_1)$ of an element $\alpha_1\in \HG^2(A_d,K_1')$ under the map $i_1^*:\HG^2(A_d,K_1')\ra\HG^2(A_d,K)$ induced from the inclusion $i_1:K_1'\hookrightarrow K$. 
The element $\alpha_1$ defines an extension $G_1$ of $A_d$ by $K_1'$ which maps via $i_1^*$ to an extension $\hat G$ of $A_d$ by $K$ that is equivalent (and hence isomorphic) to $G\leq C_2\wr_{[k]\times [d]}A_d$. Composing the inclusion $G_1\ra \hat G$ with the equivalence $\hat G\cong G$ and the projection $C_2\wr_{[k]\times [d]}A_d\ra C_2\wr_{\{2,\ldots,k\}\times [d]}A_d$, 
yields an inclusion of $G_1$ into $C_2\wr_{\{2,\ldots,k\}\times [d]}A_d$. Thus by induction, $G_1\ra A_d$ splits, and hence $\alpha_1$ and $\alpha$ are trivial. 
\end{proof}
\DN{By applying Corollary \ref{cor:induced} iteratively, we obtain  the following corollary}. For a projection $\pi:G
\ra B$ and $A\leq B$,  say that a homomorphism $s:A\ra G$ is a {\it section} of $\pi$ over $A$ if $\pi\circ s$ is the identity map on $A$. \DN{For the  projection $[S_{d_i}]_{i=1}^r\ra S_{d_1}$, we identify $S_{d_1}$ with its image under the trivial section $S_{d_1}\ra [S_{d_i}]_{i=1}^r$. 
A section $s:H\ra [S_{d_i}]_{i=1}^r$ over $H\leq S_{d_1}$ is then {\it given by conjugation} if there exists  $g\in [S_{d_i}]_{i=2}^r$,  such that $s(h) = ghg^{-1}$ for all $h\in H$. The corresponding class in  $\HG^1(H,[S_{d_i}]_{i=2}^r)$ is then $0$.} 
\begin{cor}\label{cor:split-fix}
Let $d_i\geq 5$, $i=1,\ldots,r$ be integers.   Let $G\leq [S_{d_i}]_{i=1}^r$. 
Suppose $G$ contains the composition factors of $[A_{d_i}]_{i=1}^r$. 
Then the projection $G\ra S_{d_1}$ has a section $s:A_{d_1}\ra G$ over $A_{d_1}$  that is given by conjugation. 
\end{cor}
\begin{proof}
Set $d:=d_1$. We argue by induction on $r$, with the base case $r=1$ holding trivially. Let $G_{r-1}$ be the  image of the projection $\pi:G\ra S_{d_{r-1}}\wr\cdots\wr S_{d_1}$. 
By induction the projection $[S_{d_i}]_{i=1}^{r-1}\ra S_d$ has a section over $A_d$ with image $A_d'\leq G_{r-1}$ that is given by conjugation. 

Set $m=d_2\cdots d_{r-1}$ so that $G$ acts on a set of $m\cdot d$ blocks, denoted $[m]\times [d]$. Let $\oline G:=G/A_{d_r}^{[m]\times [d]}$ be a subgroup of $[S_{d_i}]_{i=1}^r/A_{d_r}^{[m]\times [d]}=C_2^{[m]\times [d]}\rtimes [S_{d_i}]_{i=1}^{r-1}$. 
Since the quotient map $[S_{d_i}]_{i=1}^r\ra [S_{d_i}]_{i=1}^r/A_{d_r}^{[m]\times [d]}=C_2^{[m]\times [d]}\rtimes [S_{d_i}]_{i=1}^{r-1}$ has a trivial\footnote{This section is obtained from the given section of $[S_{d_i}]_{i=1}^r \ra [S_{d_i}]_{i=1}^{r-1}$ and a  section of $S_{d_r}\ra C_2$.} section $\hat s$, 
we identify $G$ with $A_{d_r}^{[m]\times [d]}\rtimes \oline G$ where the semidirect product action is through $\hat s$. 

Let $\oline \pi:\oline G\ra G_{r-1}$ be the projection, so that $\oline G\leq C_2\wr_{[m]\times [d]}G_{r-1}$ acts on  $[2]\times [m]\times [d]$ via the action of $G_{r-1}$ on $[m]\times [d]$ and that of $C_2$ on $[2]=\{1,2\}$. 
It suffices to show that $\oline \pi$ admits a section $\oline s$ over $A_d'$, given by conjugation in $C_2^{[m]\times [d]}$:
For,  $s=\hat s\circ \oline s$ is then a section of $G\ra G_{r-1}$ over $A_d'$ given by conjugation in $([S_{d_i}]_{i=2}^r)^d$. 


To construct $\oline s$, consider the kernel $\oline K= C_2^{[m]\times [d]}\cap \oline G$ of $\oline \pi$. 
As an $A_d'$-module  $C_2^{[m]\times [d]}= \Ind_{[m]
\times [d]} C_2$. 
Corollary \ref{cor:induced} then implies that 
$\oline\pi$ has a section $\oline s$ over $A_d'$. 
Identifying $\Ind_{[m]\times [d]}C_2=\Ind_{[d]}C_2^{[m]}$,  Frobenius reciprocity yields 
$\HG^1(A_d',\Ind_{[d]}C_2^{[m]})\cong \HG^1(A_{d-1}',C_2^{[m]}).$ Since $C_2^{[m]}$ is a trivial $A_{d-1}'$-module and $A_{d-1}$ has no nontrivial $2$-elementary abelian quotient, these cohomology groups are trivial. Thus,
 $\oline s$ is given by conjugation of $A_d'\leq  G_{r-1}$ in $\oline G$ by an element of $C_2^{[m']\times[d]}$. 
\end{proof}
\begin{proof}[Proof of Proposition  \ref{prop:splitting}]
Set $m:=d_2\cdots d_r$ and denote the set $[S_{d_i}]_{i=1}^r$ acts on by $[m]\times [d_1]$, where $[d_1]=\{1,\ldots,d_1\}$ is the set $S_{d_1}$ acts on. 
Assume inductively that a subgroup of $[S_{d_i}]_{i=2}^r$ containing the composition factors of $[A_{d_i}]_{i=2}^r$ contains a subgroup isomorphic to $[A_{d_i}]_{i=2}^r$. 
Let $K_1$ denote the kernel of the projection $G\ra S_{d_1}$, and consider the intersection of $K_1$ with the  copy of $S_{d_r}\wr \cdots \wr S_{d_2}$ that fixes $[m]\times \{2,\ldots,d_1\}$ pointwise. 
Since $K_1$ contains the composition factors of $(A_{d_r}\wr \cdots \wr A_{d_2})^{d_1}$, this intersection contains the composition factors of $[A_{d_i}]_{i=2}^r$, and hence by induction a copy  $\tilde A_1$ of $[A_{d_i}]_{i=2}^r$. 

\DN{On the other hand by Corollary \ref{cor:split-fix}, the projection $G\ra S_{d_1}$ has a section over $A_{d_1}$ with image $A_{d_1}'\leq G$, given by conjugation by an element in $([S_{d_i}]_{i=2}^r)^{d_1}$. Since the conjugating element is in the kernel, the  subgroup  $A_{d_1-1}'$ fixing $1$ acts on the block $[m]\times \{1\}$ trivially. 
Since in addition $\tilde A_1$ acts nontrivially only on the block $[m]\times \{1\}$ by construction,   $A_{d_1-1}'$ commutes with $\tilde A_1$. It follows that the action of $A_{d_1}'$ on conjugates of $\tilde A_1$ has a stabilizer $A_{d_1-1}'$. Thus, there are $d_1$ such conjugates 
${}^{\sigma_j}\tilde A_1$, for $\sigma_j\in A_{d_1}'$ satisfying   $\sigma_j(1)=j$, $j\in [d_1]$.} These are pairwise disjoint and centralizing. The conjugation action on them is equivalent to the action on $[d_1]$. Thus, their product $\prod_{j\in [d_1]}{}^{\sigma_j}\tilde A_1$ is normalized by $A_{d_1}'$ and is isomorphic to $\Ind_{[d_1]}\tilde A_1$.  Since this product is in the kernel of the projection to $S_{d_1}$ while $A_{d_1}'$ maps isomorphically to $A_{d_1}$, their intersection is trivial. In total, the product and $A_{d_1}'$ generate a subgroup of $G$ isomorphic to  $\Ind_{[d_1]}\tilde A_1\rtimes A_{d_1}\cong \tilde A_1\wr_{[d_1]} A_{d_1}'\cong [A_{d_i}]_{i=1}^r$. 
\end{proof}

\appendix

\section{A field version}\label{app:fields}
For the reader's convenience, we record here the ``field version" of Theorem \ref{thm:main} \DN{obtained by applying the theorem to the generic points corresponding to the fields, or alternatively,  repeating the argument when maps are replaced by field extensions}. 


The analogous terminology is as follows. 
Say that an intermediate subfield $F_1$ of an extension $F_3/F$ is Galois-proper if the Galois closure of $F_1/F$ is a proper subfield of the Galois closure of $F_3/F$. 
Say that an intermediate  subfield $F\subseteq F_1\subseteq F_3$ of an extension $F_3/F$ is {\it invariant} if the compositium $F_1\cdot F_2$ is a proper subfield of $F_3$ for every intermediate proper subfield $F_2$ of $F_3/F$. 

Let $U,V$ be the Galois groups of the Galois closures of $F_3/F_1$ and $F_1/F$, resp., so that $V$ acts on a set $J$ of cardinality $|J|=[F_1:F]$, and the Galois group $G$ of the Galois closure of $F_3/F$ embeds into $U\wr_J V$. Assume further that $\soc(U)=L^I$ is a nonabelian minimal normal subgroup of $U$, so that $G$ acts on $I\times J$ via the conjugation action on $\soc(U)^J=L^{I\times J}$ as in Proposition \ref{prop:carmel}. Let $P$ be the partition of $I\times J$ so that $\soc(U)^J\cap K \cong L^P$ where $K$ is the kernel of the projection $G\ra V$. Say that $F_1$ is {\it conjugation compatible} if $P$ and $J$ are compatible. 

\noindent {\bf Theorem \ref{thm:main}.B.} Consider a finite extension $L/F$ with an invariant  conjugation-compatible intermediate subfield $F\subsetneq F_1\subsetneq L$ such that $\soc(U)$ is a nonabelian minimal normal subgroup of the Galois group $U$ of the Galois closure of $L/F_1$. 
Assume that for every intermediate subfield $F_1'$ of $F_1/F$, $F_1$ is a Galois-proper invariant intermediate subfield of $L/F_1'$ and the Galois group of the Galois closure of $L/F_1'$ does not embed into $\Aut(\soc(U))$. Then 
the Galois group of the Galois closure of $L/F$ contains $\soc(U)^{[F_1:F]}$. 

Similarly, one may also phrase Corollary \ref{lem:consecutive} merely in terms of field extensions. Say that $E/F$ is a {\it minimal extension} if it has no proper intermediate subfields. 

\noindent {\bf Corollary \ref{lem:consecutive}.B.} $F_0=F\subseteq \cdots \subseteq F_r=E$ be a sequence of minimal extensions of degrees $d_i=[F_i:F_{i-1}]$ such that the Galois group $\Gamma_i$ of the Galois closure of $F_i/F_{i-1}$ is a nonabelian almost-simple group for $i=1,\ldots,r$. Assume that the projection from the Galois group of the Galois closure $F_{i+1}/F_{i-1}$ to $\Gamma_i$ has a nondiagonal kernel $K_i\leq \Gamma_{i+1}^{d_i}$, and that the Galois closures of $F_i/F_0$ are all distinct, for $i=1,\ldots,r-1$. Then the Galois group of the Galois closure of $E/F$ contains the composition factors of $[\soc(\Gamma_i)]_{i=1}^r$.

\end{document}